\documentclass[12pt]{amsart}
\usepackage{epsfig,color,mathtools}

\makeatother

\theoremstyle{definition}

\def\fnum{equation} 
\newtheorem{Thm}[\fnum]{Theorem}
\newtheorem{Cor}[\fnum]{Corollary}

\newtheorem{Lem}[\fnum]{Lemma}

\newtheorem{Def}[\fnum]{Definition}
\newtheorem{Exa}[\fnum]{Example}

\newtheorem{Pro}[\fnum]{Proposition}

\numberwithin{equation}{section}

\newcommand{\Vol}{{\text{Vol}}}

\newcommand{\nn}{{\bf{n}}}
\newcommand{\Ric}{{\text{Ric}}}

\newcommand{\Hess}{{\text {Hess}}}

\def\RR{{\mathbb{R}}}

\def\bV{{\bold N}}
\def\CC{{\bold C }}

\newcommand{\dv}{{\text {div}}}

\newcommand{\cC}{{\mathcal{C}}}

\newcommand{\cV}{{\mathcal{V}}}

\newcommand{\eqr}[1]{(\ref{#1})}

\title{Minimal submanifolds   confined in space}

\author{Tobias Holck Colding}%
\address{MIT, Dept. of Math.\\
77 Massachusetts Avenue, Cambridge, MA 02139-4307.}
\author{William P. Minicozzi II}%

\thanks{The  authors
were partially supported by NSF  DMS Grants   2405393 and 2304684.}

\email{colding@math.mit.edu and minicozz@math.mit.edu}

\begin{document}

\maketitle

\begin{abstract}
Already in $\mathbb{R}^4$, there are many minimal hypersurfaces but few
structural results. The half-space theorem fails there: higher-dimensional
catenoids are among the many examples contained in a slab. Nevertheless, we
show that in any dimension and codimension, the height of a minimal
submanifold controls its volume growth: confinement to a narrow cone forces
polynomial growth, and sublinear height forces Euclidean growth, the slowest
possible.  Thus volume bounds, which regularity theory assumes, follow
from height alone.

The key is a volume doubling theorem that converts geometric confinement into
quantitative rigidity: if an $n$-dimensional stationary integral varifold lies
in a thin slab at a given scale, then its volume doubling ratio at that scale
is bounded by a universal constant. Iterating this gives the polynomial growth bounds
above. Sublinear height is exactly the borderline for Euclidean growth:
examples  grow at least like $r^{n+4\epsilon^2/5}$ in
cones of aperture $\epsilon$. From Euclidean volume growth, we obtain an
optimal Bernstein theorem, in all dimensions, for stable minimal hypersurfaces
with sublinear height growth, generalizing results of Moser, Bombieri--De
Giorgi--Miranda, Trudinger, Caffarelli--Nirenberg--Spruck, and Ecker--Huisken.
\end{abstract}


\section{Introduction}

  De Giorgi and Nash independently resolved Hilbert's 19th problem in the late 1950s by establishing H\"older continuity for $L^2$
  weak solutions of elliptic PDEs. A key ingredient in De Giorgi's approach is the control of a scale-invariant $L^2$ 
 deviation of the graph of a function from an affine plane, known as the height-excess. He showed that this quantity decays at smaller scales \cite{dG,Mo1}.

In subsequent developments, Almgren (in the 1960s) and Allard (in the 1970s) extended this perspective to a coordinate-free setting for general sets that solve an elliptic equation weakly. Rather than working with graphs and deviation from a plane directly, they introduced a stronger varifold-based condition, which for functions corresponds roughly to a $W^{1,2}$ 
  bound. Allard termed the associated quantity the tilt-excess \cite{Al}. However, many applications only provide bounds closer in spirit to De Giorgi's original $L^2$ 
  control, and it is this weaker framework that we adopt here. 
 
The tilt-excess bound implicitly incorporates a scale-invariant volume bound, and all existing arguments rely crucially on a volume bound.
Here we show that the weaker height-excess bound already implies a volume bound. This leads to many applications.

The central ingredient is the volume doubling theorem. 
It shows that a stationary integral varifold confined to a sufficiently thin slab at a given scale has a uniform doubling bound. 
 This converts the weak geometric information encoded by the height-excess into the measure-theoretic information required by the classical regularity theory. 
 
 \subsection{Doubling}
The following theorem highlights the link to De Giorgi's initial bound on the height-excess, rather than a stronger bound involving the tilt-excess. It establishes a volume-doubling bound for minimal submanifolds whose height is small at a fixed scale.

\begin{Thm}	\label{t:loosedoubling}
There exists $\delta_0 > 0$   depending only on $n$, so that if  $\Sigma^n \subset \RR^{n+k}$ is a proper stationary integral varifold and $B_{6r} \cap \Sigma$ is contained in a slab{\footnote{The center of the ball must lie in the center of the slab.  See Corollary \ref{c:doubling}  for a stronger statement.}}
 of height $\delta_0 \, r$, then 
\begin{align}
	\Vol ( B_{2r} \cap \Sigma) \leq n^4 \, 2^{4n} \, \Vol (B_r \cap \Sigma) \, .
\end{align}
\end{Thm}

      \vskip1mm
 This theorem is scale-invariant and applies whenever the submanifold lies in a slab
  at each scale, where the slab itself may vary from one scale to another. This flexibility allows for fractal behavior. If the slab condition holds at all sufficiently large scales, then the theorem can be iterated to obtain a polynomial volume bound.  Theorem \ref{t:loosedoubling} is  sharp; see Subsection \ref{ss:3regimes} below.
   
  \subsection{Euclidean volume growth}
   Monotonicity for proper\footnote{An immersion is proper if its intersection with every compact set is compact. Similarly, a
    stationary varifold is proper if its density is bounded on compact sets and the  intersection of its support with every compact set is compact.} minimal immersions $\Sigma^n \subset \mathbb{R}^{n+k}$ implies that the density
\begin{align}    \label{e:density}
r^{-n}\,\mathrm{Vol}(B_r(x)\cap \Sigma)
\end{align}
is nondecreasing in $r$.  As $r\to 0$, the density converges to the volume of the unit ball in $\RR^n$ 
  if $x\in \Sigma$ and $\Sigma$ is embedded; it converges  to $0$ when $x\notin \Sigma$. 
  The density 
   is the most fundamental quantity in the theory of minimal surfaces, playing crucial roles in both regularity theory and global structure. 
  
 \begin{Def}
  A minimal submanifold is said to have Euclidean volume growth if the limit as $r \to \infty$
  of the density is finite.
  \end{Def}

   An $n$-dimensional submanifold $\Sigma^n\subset \RR^{n+k}$ is said to have sublinearly growing height if, after a rigid motion, there exists $\alpha<1$ and $r_0>0$ such that
   \begin{align}
   \Sigma\setminus B_{r_0}\subset \{x\in\RR^{n+k}\,|\,x_{n+1}^2+\cdots+x_{n+k}^2<(x_1^{2}+\cdots+x_n^{2} )^{\alpha}\}\,  .
   \end{align}
   We will also consider weak solutions of the minimal surface equation known as stationary integral varifolds, \cite{dL,Si}; this allows for singularities.  We will say that a stationary integral varifold has sublinearly growing height if its support does.
   \vskip1mm
            
     \begin{Thm}		\label{t:evg}
     A   proper stationary integral varifold in any dimension and codimension whose height	 grows sublinearly  must have Euclidean volume growth.  
     \end{Thm}

     Euclidean volume growth is highly restrictive. For minimal surfaces in $\RR^3$
      of finite genus, it is equivalent to finite total curvature, and such surfaces have long been well understood \cite{JM,Os}. Complex submanifolds of $\CC^N$, which are minimal by work of Wirtinger \cite{Wi,F},
       are algebraic{\footnote{A complex submanifold in $\CC^N$ is algebraic if it is the zero set of a finite collection of complex polynomials.}}
       if they have Euclidean volume growth \cite{B,R,S}. Without such a density bound, much less is known.
       The main exception is the class of embedded minimal surfaces in $\RR^3$
    \cite{CM2,CM3,CM4,CM5,CM6,MPR,P1,P2,BB1,BB2}.
       Most minimal submanifolds do not have Euclidean volume growth. 
       Helicoids  \cite{CjH}        and Riemann examples    \cite{KPa} 
 in any dimension grow polynomially. 
       There are also proper minimal surfaces in $\RR^3$ and $\RR^4$, area-minimizing in $\RR^4$, whose area grows exponentially or faster with no upper limit on the rate \cite{CMM}; products with $\RR$
        give such examples in all dimensions.

The hypothesis of Theorem  \ref{t:evg}  holds for many natural examples. The catenoid has logarithmic height in $\RR^3$ and bounded height in higher dimensions \cite{Bl}, and Enneper’s surface has height growing like $ r^{2/3}$. The embedded finite-topology surfaces in 
$\RR^3$ constructed by Costa, Jorge, Hoffman, Meeks, Karcher, Wei, Kapouleas, Traizet and others all have at most logarithmic height growth \cite{HK,MP}. Fakhi and Pacard constructed immersed analogues in 
higher dimensions \cite{FPa}. Generalized Enneper surfaces \cite{Cj} and Kapouleas’s desingularizations \cite{K}, which have arbitrarily many ends, show that the 
constant in Theorem  \ref{t:evg}  can be arbitrarily large. Our proof shows that this constant is determined by the volume in a fixed ball whose radius depends only on $r_0,\alpha$ and $n$.
Theorem  \ref{t:evg},  like the other results here, has a version allowing compact boundary; these are stated in the body of the paper. Sublinear growth cannot be relaxed to linear growth with small slope; see Subsection \ref{ss:3regimes}.

\subsection{Stable Bernstein in all dimensions} 
Using Theorem \ref{t:evg}, we  show the following optimal Bernstein theorem  in all dimensions for stable minimal hypersurfaces 
with sublinearly growing height:

\begin{Thm}	\label{t:SSB}
Suppose that $\Sigma^n \subset \RR^{n+1}$ is a complete properly immersed{\footnote{Theorems \ref{t:SSB},
 \ref{t:SSB2} and \ref{t:LCE} hold  if $\Sigma$ has a singular set with locally finite codimension two measure; cf. \cite{Be,SS,W}.}}
 stable minimal hypersurface.  If $\Sigma$ is two-sided and has sublinearly growing height,
then $\Sigma$ is a hyperplane.
\end{Thm}

The sublinear growth is a necessary condition, as stable minimal cones - such as the Simons cone in $\RR^8$ -  exhibit exactly linear height growth.

Theorem \ref{t:SSB} generalizes several classical Bernstein-type results for minimal graphs, starting with Moser's theorem \cite{Mo2}. Moser's work was subsequently extended by Bombieri--De Giorgi--Miranda \cite{BDM}, Trudinger \cite{Tr}, Caffarelli--Nirenberg--Spruck \cite{CfNS}, and Ecker--Huisken \cite{EH}; see also \cite{CGMR, CMMR, D1, D2}.

The proof of Theorem \ref{t:SSB} uses the volume bound from Theorem \ref{t:evg} together with Bellettini's $\epsilon$-regularity for the tilt, \cite{Be,SS}.  We will  prove  a generalization of this that allows for a compact boundary in Theorem
 \ref{t:SSB2}.    We will also show a stronger curvature estimate (see Theorem \ref{t:LCE}) for stable minimal hypersurfaces with boundary that implies Theorem \ref{t:SSB}.

\subsection{Slabs}
 The  half-space theorem of Hoffman-Meeks, \cite{HM} (cf. \cite{CM7}), shows that in $\RR^3$ planes are the only proper minimal surfaces in a half-space.  In higher dimensions, the situation is quite different.  For $n>2$, there exist many minimal hypersurfaces contained in slabs. For example, higher-dimensional catenoids lie within a slab \cite{Bl}. 
 
 For a stationary integral varifold 
 \begin{align}	\label{e:slabref}
 	\Sigma \subset \{x \in \RR^{n+k} \mid \sum_{j=1}^k x_{n+j}^2 \leq 1\}
\end{align}
 in a slab, we obtain stronger estimates and precise asymptotics.  We show that  $\Sigma$ not only exhibits Euclidean volume growth in every dimension, but also satisfies an optimal bound on the rate at which its volume approaches this growth:

    \begin{Thm}	\label{t:rate}
There exist $ C$ and $R_0$ depending only on $n$ so that if 
$\Sigma^n$ is a proper stationary integral varifold satisfying \eqr{e:slabref}, then there exists an integer $\bV$ so that
for all $r \geq R_0$ and all $p$ with $p_{n+1} = \dots = p_{n+k}= 0$
\begin{align}		\label{e:rate}
	(1-C\,  r^{-2}) \, \bV   \leq  \frac{\Vol (B_r (p)  \cap \Sigma)}{ \Vol (B_r \subset \RR^n)}  \leq \bV \, .
\end{align}
 \end{Thm}
 
 \vskip1mm
 The key is a weighted integrability for the tilt-excess proven below in Theorem \ref{t:slab2}.
 The properness assumption  is necessary as there are non-proper minimal immersions contained in a slab, even in $\RR^3$,  with infinite area already in the unit ball, \cite{JX}.

 The factor $r^{-2}$ on the left-hand side of \eqr{e:rate} is sharp as can  be seen in the simple example where $\Sigma$ is a hyperplane  
 $\{ x_{n+j} = \zeta_j \, | \, j=1 , \dots , k\}$ with $\sum_{j=1}^k \zeta_j^2 = 1$  and we have
\begin{align}
 	r^n - \frac{ \Vol (B_r   \cap \Sigma)} {\Vol (B_1 \subset \RR^n)} \approx \frac{n}{2} \, r^{n-2} \, .
\end{align}
This example also shows that the constant $C$  in \eqr{e:rate} depends on $n$.

 \subsection{Stable surfaces in $\RR^4$}
 
 Every holomorphic curve
in $\CC^n= \RR^{2\,n}$ is absolutely area minimizing with respect to arbitrary compactly supported deformations \cite{Wi,F}. 
Micallef proved a converse of this for $\RR^4$  in \cite{M}.  Namely,  he showed that any complete oriented parabolic stable minimal surface in $\RR^4$ is holomorphic
with respect to some orthogonal complex structure on $\RR^4$.   Combining this theorem with our results gives the following:

\begin{Thm}	\label{t:mario}
Suppose that $\Sigma^2 \subset \RR^4$ is an oriented stable stationary 
integral varifold with at most finitely many singular points. If $\Sigma$ is contained in a sublinearly growing tubular 
neighborhood of some two-plane, then it is a complex curve{\footnote{In fact, $\Sigma$ must be algebraic by \cite{R}; cf. \cite{ERM,S}.}}
 for some orthogonal complex structure on $\RR^4$.
\end{Thm}

\begin{proof}
The sublinear growth of the height allows us to apply Theorem \ref{t:EVGn2} to conclude that $\Sigma$ has quadratic area growth.  This implies parabolicity by the standard logarithmic cutoff argument. Thus,   \cite{M}  applies to give the claim; cf. \cite{ERM}.
\end{proof}

There are many examples of holomorphic curves satisfying the hypotheses in Theorem \ref{t:mario}. 
 For example, using complex coordinates $(z_1 , z_2)$ on $\RR^4 = \CC^2$, 
any set of the form $z_1^p = z_2^q$ for positive integers $p \ne q$ will be stable, oriented, and contained
 in a sublinearly growing tubular neighborhood of either $z_1 = 0$ or $z_2 = 0$ (depending on the ratio  $p/q$).

\subsection{One-sided bounds}
The previous results assume  bounds for  $\Sigma^n \subset \RR^{n+k}$ from two sides. 
	The next theorem generalizes this to hypersurfaces  on one side of a hyperplane and, more generally, a higher codimension version of this. 
	Define cylinders
\begin{align}
	\cC_r = \left\{ \sum_{i=1}^n x_i^2 < r^2 \right\} 
\end{align}
and truncated cylinders $\cC_{r,b} = \cC_r \cap \{ x_{n+j} < b , \, j = 1 , \dots , k\}$. 

\begin{Thm}	\label{t:boundingoutB}
There exist    $C_1 , C_2 , C_3 > 0$   depending on $n$ and $k$ so that if $\Sigma^n \subset \RR^{n+k}$ is a proper stationary integral varifold, 
 $z_0 > 0$ satisfies  $r > C_1 \, z_0  $,
and 
$$\cC_{4\,r} \cap \Sigma \subset \{ x_{n+j} > 0 , \, j=1 , \dots , k \},$$   then
\begin{align}
	\Vol \, ( \cC_{2 \, r ,  \, C_2 \, z_0} \cap \Sigma) \leq C_3 \, \Vol \, ( \cC_{  r, z_0} \cap \Sigma) \, .
\end{align}
\end{Thm}

 \begin{figure}[htbp]
\includegraphics[scale=0.55]{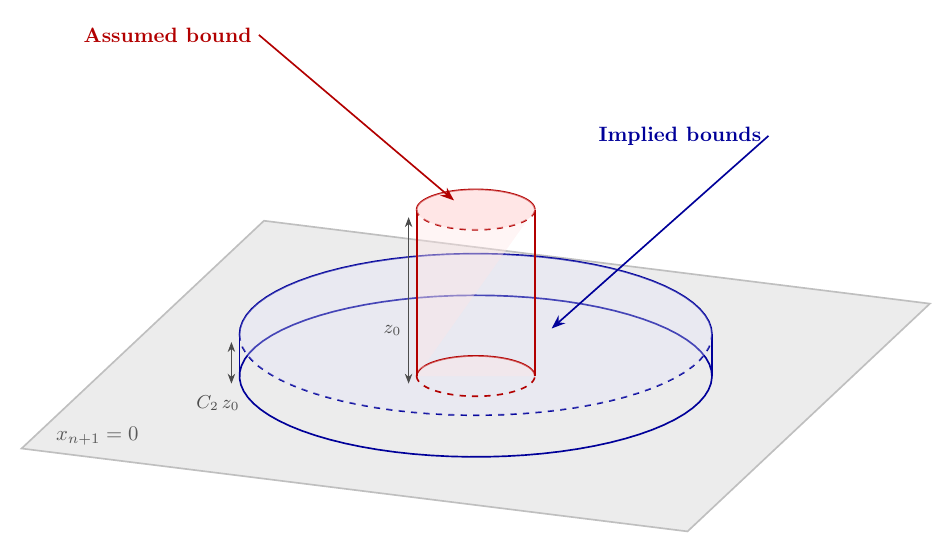}
\end{figure}

The theorem gives a volume doubling as we go further out, but at the cost of decreasing the height of the slab. 
Simple examples show that it is necessary to decrease the height   to get a doubling bound in a larger cylinder.

\subsection{Sharpness} \label{ss:3regimes}
 Theorems  \ref{t:loosedoubling} and \ref{t:evg}
  prove upper bounds on volume growth using analytic estimates. Whereas \cite{CM9} constructs examples showing that these estimates are sharp and analyzes the complex case.
 
There are three regimes. 
 Suppose 
 $\Sigma^n\subset\RR^{n+k}$ is confined  to $\{|\Pi_2(x)|\le\epsilon\,|\Pi_1(x)|+K\}$, a thickened cone of aperture $\epsilon$
around an $n$-plane.
\begin{itemize}
\item Sublinear height forces Euclidean volume growth (Theorem \ref{t:evg}).
\item In a sufficiently narrow cone, Theorem \ref{t:loosedoubling} gives polynomial growth.
When the cone aperture $\epsilon > 0$ is small, our methods  show that
the  rate is at most $n+O(\epsilon^2)$.  This goes to the Euclidean rate as $\epsilon \to 0$.  Examples in \cite{CM9} show that 
   $n+O(\epsilon^2)$ is
 sharp.
\item In cones of aperture at least $\max_t t/\cosh t\approx0.66$, sums of
catenoids have arbitrarily fast growth \cite{CM9}.
\end{itemize}
The exponent $n+O(\epsilon^2)$ is sharp: already in $\RR^3$, for every small
$\epsilon>0$, there is a stationary integral varifold in a cone of aperture $\epsilon$
whose area grows at least like $r^{2+\frac45\epsilon^2}$ \cite{CM9}.  Thus sublinear
height is exactly the borderline for Euclidean growth, and linear height does not
suffice for stationary varifolds no matter how small the angle.  Holomorphic curves in $\CC^2$ behave
differently.  For them linear height does suffice, below a sharp aperture
$\cot(\theta/2)$ depending on the K\"ahler angle $\theta$ of the plane, and the curve
is then algebraic \cite{CM9}.

\subsection{Organization}

In Sections \ref{s:S1} and \ref{s:S2}, we will show that a  proper minimal immersion whose height grows sublinearly  must have Euclidean volume growth in all dimensions and any codimension.  This is accomplished by first showing a volume doubling from one scale to the next in Section \ref{s:S1} and then obtaining the stronger result in Section \ref{s:S2}.  
 Section \ref{s:S4} contains
  the proof of the  Bernstein theorem for proper minimal hypersurfaces contained in a Euclidean slab in any dimension. 
The main result in Section \ref{s:S3}  is   a rate of convergence for the asymptotic volume ratio at infinity.
  Section \ref{s:S5} turns to the case where 
  the minimal submanifold is in a half-space, but there is no restriction on its growth upwards.  We again prove a volume doubling, but with the caveat that we must look closer to the boundary as we go further out.  We  focus on submanifolds of dimension three and above in Sections \ref{s:S1}--\ref{s:S5}.  In Section \ref{s:S6}, we turn to the case of surfaces where more is known and parabolicity often leads to more restrictive results.
  
\vskip2mm
The results of this paper are used in \cite{CM8} to prove Liouville theorems for minimal surfaces  in high codimension
and were the starting point for \cite{CM9}. These results were surveyed in \cite{CM10}.

\section{Volume doubling}	   \label{s:S1}

We will show that if a minimal submanifold $\Sigma^n \subset \RR^{n+k}$ is  contained in a scale-invariantly small  narrow tubular neighborhood of an 
$n$-plane on some fixed scale, then there is a uniform volume doubling on this scale, independent of the codimension $k$. 
We will restrict to $n\geq 3$ in the next five sections and turn to the case of surfaces in Section \ref{s:S6}.

\vskip1mm
To state the volume doubling, 
let  $\Pi_1$ be orthogonal projection from $\RR^{n+k}$ to the first $n$ coordinates,
  set   $\Pi_2(x) = x - \Pi_1(x)$, and define $s \geq 0$ by
	$s^2 = |\Pi_1 (x)|^2$. 
 Given $\delta > 0$, define  solid cones  
  \begin{align}
 	\cC_{\delta} = \{ x \, | \, |\Pi_2 (x)| \leq \delta \, |\Pi_1 (x)| \} \, .
 \end{align}
 The $n$-plane where $\Pi_2 (x) = 0$ is the intersection of all the $\cC_{\delta}$'s with $\delta > 0$. 
 Given   $R> 0$, 
 let  
 	$V(R) = \Vol (\{ s < R \} \cap \Sigma) $
denote the $n$-dimensional volume of $\Sigma$ in the cylinder $\{ s < R \}$. 
Given $a< b$, let $\Sigma_{a,b} = \{ a < s < b \} \cap \Sigma$.  

\vskip1mm
 The next theorem gives the doubling bound in
Theorem \ref{t:loosedoubling} with an explicit constant.

\begin{Thm}  \label{t:doubling}
There exists $\delta_0 > 0$ depending only on $n$ so that if 
 $\Sigma^n \subset \RR^{n+k}$ is a proper stationary integral varifold  and 
\begin{align}	\label{e:1p3}
	\Sigma_{\frac{R}{2}, 4R  } \subset \{ |\Pi_2 (x)| \leq \delta_0 \, R\} \, , 
\end{align}
then 
 $ V(2R) \leq n^2  \, 2^{2\, n} \, V(R) $.
 \end{Thm}
 
 Iterating the volume doubling bound in Theorem \ref{t:doubling} implies at most polynomial volume growth
 when $\Sigma$ is complete and contained in the double cone $\cC_{\frac{\delta_0}{4}}$.

\vskip1mm
Combining Theorem \ref{t:doubling} with a standard covering argument gives the following doubling bound when $\Sigma$ is allowed to have boundary 
or singularities in an inner ball:

\begin{Cor}  \label{c:doubling}
There exist $\delta_1 > 0$ and $C_n$ depending only on $n$ so that if 
 $\Sigma^n$ is a proper stationary integral varifold in $\{ \frac{R}{4} < s < 4 \, R \} \subset \RR^{n+k}$ and 
\begin{align}	\label{e:1p3c}
	  \Sigma \subset \{ |\Pi_2 (x)| \leq \delta_1 \, R\} \, , 
\end{align}
then 
 $  \Vol (\Sigma_{R,2R} ) \leq C_n \, \Vol (\Sigma_{ \frac{R}{2} , R})$.
 \end{Cor}

 \subsection{General facts for stationary varifolds}
 
 The next lemma contains  general estimates for a stationary integral varifold $\Sigma^n \subset \RR^{n+k}$. 

\begin{Lem}	\label{l:sdoesk}
On $\Sigma$, we have that
\begin{align}
	|\nabla^{\perp} s|^2 & \leq  
	\sum_{j=1}^k   | \nabla^T x_{n+j}|^2 \, , \\
	-(n-1) \, (n-2) \, \sum_{j=1}^k   | \nabla^T x_{n+j}|^2 &\leq s^n \, \dv_{\Sigma} \, (\nabla  \, s^{2-n} ) \leq (n-2) \, \sum_{j=1}^k   | \nabla^T x_{n+j}|^2 \, .
\end{align}
\end{Lem}

\vskip1mm
The quantity $\sum_{j=1}^k |\nabla^T x_{n+j}|^2$ has a natural geometric interpretation as the squared deviation from $T_x\Sigma$
to the $n$-plane $\Pi_2 (x) = 0$  and its average is a rough measure of flatness known as the tilt-excess (\cite{Al}, cf. \cite{dL} and chapter $5$ in \cite{Si}).  

\begin{proof}[Proof of Lemma \ref{l:sdoesk}]
Fix a point $x \in \Sigma$ and let $\nn_1 , \dots \nn_k$ be an orthonormal frame for the normal space to $\Sigma$ at $x$. Observe that
\begin{align}	\label{e:nabsk}
	|\nabla^{\perp} s|^2 &=  \frac{ | (\Pi_1 (x))^{\perp}|^2}{|\Pi_1 (x)|^2} = \sum_{i=1}^k  \frac{ \langle \Pi_1 (x) , \nn_i \rangle^2}{|\Pi_1 (x)|^2} 
	=  \sum_{i=1}^k  \frac{ \langle \Pi_1 (x) , \Pi_1(\nn_i) \rangle^2}{|\Pi_1 (x)|^2} \notag \\
	&\leq  \sum_{i=1}^k | \Pi_1(\nn_i)|^2 = \sum_{i=1}^k \left(  1 - |\Pi_2 (\nn_i)|^2 \right) = k - \sum_{i=1}^k   |\Pi_2 (\nn_i)|^2  \\
	&= k - \sum_{i,j=1}^k   \langle \nn_i , \partial_{n+j} \rangle^2  = k -  \sum_{j=1}^k  (1 - |\partial_{n+j}^T|^2) = 
	\sum_{j=1}^k   |\partial_{n+j}^T|^2 = 
	\sum_{j=1}^k   | \nabla^T x_{n+j}|^2
	 \, . \notag
\end{align}
This gives the first claim. 

We will prove the second claim  in the case where $\Sigma$ is smooth; the arguments are pointwise and the varifold versions hold almost everywhere by standard arguments.
Since $\Sigma$ is minimal, $\dv_{\Sigma}\nabla  |x|^2 = 2n$ and $\dv_{\Sigma} \nabla x_i = 0$ (see, e.g., $(1.57)$ and 
proposition $1.7$ in \cite{CM1}), so 
\begin{align}	\label{e:110e}
	\dv_{\Sigma} \, \nabla s^2 = 2n - \dv_{\Sigma} \, \nabla  |\Pi_2 (x)|^2 = 2n - 2\,  
	\sum_{j=1}^k   | \nabla^T x_{n+j}|^2 \, .
\end{align}
Using this and the chain rule, we get that 
\begin{align}
	s^n \, \dv_{\Sigma} \, \nabla s^{2-n} &=  (1-n/2) \, s^n \, \dv_{\Sigma} \, \left( s^{-n} \, \nabla^T s^2 \right)\notag \\
	& = 
	(1-n/2) \, s^n \,   \left( s^{-n} \,\dv_{\Sigma} \, \nabla s^2 - \frac{n}{2} \, s^{-n-2} \,  |\nabla^T s^2|^2 \right) \\
	&=   (n-2) \,    \left(       \sum_{j=1}^k   | \nabla^T x_{n+j}|^2  -n \, |\nabla^{\perp}  s|^2 \right) \, . \notag
\end{align}
The right-hand side is at most $(n-2) \,     \sum_{j=1}^k   | \nabla^T x_{n+j}|^2 $. Moreover, using \eqr{e:nabsk}, it is
 at least $(n-2) (1-n) \, \sum_{j=1}^k   | \nabla^T x_{n+j}|^2$.  
\end{proof}

As a consequence, we see that the tilt-excess controls the integral of $|\nabla^T s|^2$:

\begin{Lem}	\label{l:tiltgives}
If $\Sigma^n \subset \RR^{n+k}$ is as in Theorem \ref{t:doubling}, then 
\begin{align}
	\int_{\Sigma_{2R}} |\nabla^T s|^2 \leq n^2 \, 2^{2\, n -1} \, V(R) + (n-1) \, 2^{n-1} \, \int_{\Sigma_{R,2R}} 
	 \sum_{j=1}^k |\nabla^T x_{n+j}|^2
	\, .
\end{align}
\end{Lem}

\begin{proof}
Let $\phi (s)$ be a piecewise linear function that is $0$ for $s\leq \frac{R}{2}$, has $\phi (R) = 1$, and is zero for $2R \leq s$. 
In particular $\phi'(s) = \frac{2}{R}$ on $[R/2, R)$ and $\phi'(s) = - \frac{1}{R}$ on $(R, 2R]$. 

Applying the divergence theorem to $\dv_{\Sigma} \, (\phi (s) \, \nabla s^{2-n})$ gives 
\begin{align}
	0 &= \int_{\Sigma} \phi'(s) \, \langle \nabla^T s , \nabla s^{2-n} \rangle + \phi (s) \, \dv_{\Sigma} (\nabla s^{2-n}) \notag \\
	&=  \frac{2}{R} \, (2-n) \,  \int_{\Sigma_{\frac{R}{2},R}} s^{1-n}  \,| \nabla^T s |^2 
	-\frac{1}{R} \, (2-n) \,  \int_{\Sigma_{R,2\,R}} s^{1-n}  \,| \nabla^T s |^2
	 + \int_{\Sigma} \phi (s) \, \dv_{\Sigma} (\nabla s^{2-n}) \, , \notag
\end{align}
where $\Sigma_{a,b}$ denotes $\{ x \in \Sigma \, | \, a < s(x) < b \}$. 
Combining this and 
Lemma \ref{l:sdoesk} gives
\begin{align}
	  \int_{\Sigma_{R,2\,R}} s^{1-n}  \,| \nabla^T s |^2 &\leq 2\, \int_{\Sigma_{\frac{R}{2},R}} s^{1-n}  \,| \nabla^T s |^2 
	  + (n-1) \, R \,  \int_{\Sigma_{ \frac{R}{2} , 2R} } s^{-n} \, \sum_{j=1}^k |\nabla^T x_{n+j}|^2  \notag \\
	  &\leq 2^n \,  R^{1-n} \, V(R)  
	  + (n-1) \, R^{1-n} \, \left(  n \, 2^n \, V(R) +  \int_{\Sigma_{ R , 2R} }  \sum_{j=1}^k |\nabla^T x_{n+j}|^2 \right) \, . \notag
\end{align}
Simplifying this gives the claim.
\end{proof}

\subsection{From one scale to the next}

Suppose that $\Sigma^n \subset \RR^{n+k}$ is as in Theorem \ref{t:doubling} and, thus, lies in a tubular neighborhood of the $n$-plane   $\{ x \, | \, \Pi_2 (x) = 0\}$.   In the model case where $\Sigma = \{ x \,  | \, \Pi_2(x) = 0 \}$, the function $s^{2-n}$ is harmonic on $\Sigma$.  The next lemma generalizes this, showing that we can add a multiple of 
$|\Pi_2 (x)|^2$ to $s^{2-n}$ to get a subharmonic function $h$ 
\begin{align}	\label{e:defh114}
	h = s^{2-n} + 2^{n-1} \, n \, (n-2)  \, \frac{|\Pi_2 (x)|^2}{R^n}  \, .
\end{align}
The next lemma shows that $h$ closely approximates $s^{2-n}$ and $h$ is subharmonic.

\begin{Lem}	\label{l:gg}
Given  $\epsilon \in (0,1/100)$, there exists $\delta_0= \left( \frac{\epsilon \, 2^{5-3\, n}}{n^2}\right)^{ \frac{1}{2}} > 0$  so that if  
  $|\Pi_2 (x)| \leq \delta_0 \, R$   in the region $R/2 \leq s \leq 4 R$, then in this region
\begin{align}
	\epsilon \, s^{2-n} & \geq |s^{2-n} - h |    \, , \\
	\dv_{\Sigma} \, (\nabla h) &\geq 2^n \, (n-2) \,  \frac{\sum_{j=1}^k   | \nabla^T x_{n+j}|^2}{ R^n} \, .
\end{align}
\end{Lem}

\begin{proof}
In this region,  we see that
\begin{align}
	|s^{2-n} - h | &=  2^{n-1} \, n\, (n-2) \, \frac{|\Pi_2(x)|^2}{R^n} \leq  2^{n-1} \, n\, (n-2) \, \delta_0^2   \,  R^{2-n} 
	  \leq   2^{3\, n-5} \, (n^2-2n) \, \delta_0^2   \, s^{2-n} \, . \notag
\end{align}
This gives the first claim.   The second claim follows from Lemma \ref{l:sdoesk} 
\begin{align}
	\frac{\dv_{\Sigma} \, \nabla h}{n-2}  & = \frac{\dv_{\Sigma} \, \nabla s^{2-n}}{n-2} +  2^{n-1} \, n   \, \frac{\dv_{\Sigma} \, \nabla |\Pi_2(x)|^2}{R^n} \geq   \left(   (1-n)  \, \frac{R^n}{s^n}
	+ 2^{n} \, n  \right)  \, \frac{ \sum_{j=1}^k   | \nabla^T x_{n+j}|^2}{R^n} \notag \\
	&\geq    2^n \, \frac{\sum_{j=1}^k   | \nabla^T x_{n+j}|^2}{R^n}
	 \, . 
\end{align}
\end{proof}

The function $h$ is decreasing and  $\dv_{\Sigma} \nabla h$ is bounded below by  the tilt-excess.  
Heuristically, if we integrated $\dv_{\Sigma} \nabla h$ over an ``annulus'',  we would bound
the tilt-excess in terms of the inner boundary 
  (the  contribution from the outer boundary has a favorable sign).  Since we are working with varifolds, we cannot quite argue this way. 
The next lemma uses the   function $h$  to bound 
$\sum_{j=1}^k \int_{ \{ R \leq s \leq 2\, R\}} |\nabla^T x_{n+j}|^2$ in terms of $V(R)$.  This gives tilt-excess control 
 on a larger set in terms of the volume on a smaller set.  This would be clear if we  had a density bound on the larger set, but we do not have this here. In fact, the lemma
   will be crucial for controlling the growth of the volume.

\begin{Lem}	\label{l:partway}
If $\Sigma^n \subset \RR^{n+k}$ is as in Theorem \ref{t:doubling}, then 
\begin{align}
	 \int_{\Sigma_{R,2\, R}} \sum_{j=1}^k   | \nabla^T x_{n+j}|^2 &\leq
		  2^{n+1} \, V(R)
	    \, .
\end{align}
\end{Lem}

\begin{proof}
Define $\bar{s}$ in terms of $h$ by $h = \bar{s}^{2-n}$. 
Since $h$ approximates $s^{2-n}$, we have that $\bar{s}$ approximates $s$. In particular, we have that
\begin{itemize}
\item $\Sigma_{R,2R} \subset \{ \frac{7}{8}\, R < \bar{s} < \frac{9}{4} \, R \}$. 
\item $\{ \frac{5}{8} \, R < \bar{s} < \frac{7}{8} \, R \} \subset \{ \frac{R}{2} < s < R\}$. 
\item $\{ \frac{9}{4} \,  R < \bar{s} < \frac{11}{4} \, R \} \subset \{ 2\,R < s < 3\, R\}$.
\end{itemize}
Let $\phi (s)$ be a piecewise linear function that is zero at $\frac{5}{8} \, R$, then one on $[7R/8, 9R/4]$, and then zero on $\frac{11}{4} \, R$.   Applying the divergence theorem to $\dv_{\Sigma} \, ( \phi (\bar{s}) \, \nabla h)$ gives 
\begin{align}
	0 &= \int_{\Sigma} \dv_{\Sigma} \, ( \phi (\bar{s}) \, \nabla h) = \int_{\Sigma} 
	\left( \phi'(\bar{s}) \, \langle \nabla^T \bar{s} , \nabla h \rangle + \phi (\bar{s}) \, \dv_{\Sigma} (\nabla h) \right) \, .
\end{align}
Using Lemma \ref{l:gg} on the last term, the fact that $\phi$ is one on $\Sigma_{R,2R}$, and the chain rule
$\nabla \bar{s} = - \frac{\bar{s}^{n-1} \, \nabla h}{n-2}$, we see that 
\begin{align}
	2^n \, (n-2) \, R^{-n} \, \int_{\Sigma_{R,2\, R}} \sum_{j=1}^k   | \nabla^T x_{n+j}|^2 &\leq
		\frac{1}{n-2} \, \int_{\Sigma} \bar{s}^{n-1} \, 
	 \phi'(\bar{s}) \, | \nabla^T  h|^2  \, .
\end{align}
Since $\phi'$ is negative on the outer region and is $\frac{4}{R}$ on the inner region which is contained in $\Sigma_{\frac{R}{2},R}$, we see that 
\begin{align}	\label{e:partwayT}
	2^n \, (n-2)^2 \,   \int_{\Sigma_{R,2\, R}} \sum_{j=1}^k   | \nabla^T x_{n+j}|^2 &\leq
		4\, R^{2n-2} \,  \int_{\Sigma_{\frac{R}{2},R}}  
	   | \nabla^T  h|^2  \, .
\end{align}
  We have that 
\begin{align}
	\nabla^T  |\Pi_2(x)|^2 = 2 \, (\Pi_2 (x))^T = 2 \, \sum_{j=1}^k x_{n+j} \, \partial_{n+j}^T \notag
\end{align}
and, thus, 
\begin{align}
	 \frac{1}{4} \, | \nabla^T  |\Pi_2(x)|^2  |^2  \leq    |\Pi_2 (x)|^2 	   \, . 
\end{align}
Using this, 
the squared triangle inequality $(a+b)^2 \leq \frac{3}{2} \, a^2 + 3\, b^2$ gives on this region that
\begin{align}
	|\nabla^T h|^2 &\leq \frac{3}{2}  \,(n-2)^2 \, s^{2-2n} \,  |\nabla^T s  |^2 + 3 \,  \frac{2^{2n} \, n^4}{R^{2n}} \,  | \nabla^T  |\Pi_2(x)|^2  |^2 \notag \\
	&\leq 3\, 2^{2n-3} \, (n-2)^2 \, R^{2-2n}   +  3\, \delta_0^2 \, 2^{2n+2} \, n^4 \, R^{2-2n}  \leq (n-2)^2 \, 2^{2n-1} \, R^{2-2n} 
	 \, ,
\end{align}
where the last inequality assumes that $\delta_0 > 0$ is small (depending just on $n$).  Using this in \eqr{e:partwayT} gives 
\begin{align}	\label{e:partwayT2}
	2^n \, (n-2)^2 \,   \int_{\Sigma_{R,2\, R}} \sum_{j=1}^k   | \nabla^T x_{n+j}|^2 &\leq
		4\, R^{2n-2} \, V(R) \, \left(  (n-2)^2 \, 2^{2n-1}\, R^{2-2n}  \right)
	    \, .
\end{align}
Simplifying this gives the lemma.
\end{proof}

\begin{proof}[Proof of Theorem \ref{t:doubling}]
To shorten notation, let $E= E(x) =    \sum_{j=1}^k |\nabla^T x_{n+j}|^2$  be the tilt-excess at each point.
Lemma \ref{l:tiltgives}
gives that
\begin{align}
	\int_{\Sigma_{2R}} |\nabla^T s|^2 \leq n^2 \, 2^{2\, n -1} \, V(R) +  (n-1) \, 2^{n-1} \, \int_{\Sigma_{R,2R}} 
 E	\, .
\end{align}
The first claim in Lemma \ref{l:sdoesk}
gives that
	$|\nabla^{\perp} s|^2 
	 \leq  
	E$, so we see that
\begin{align}
	V(2R) &= \int_{\Sigma_{2R}} \left(  |\nabla^T s|^2 + |\nabla^{\perp} s|^2  \right) \leq \int_{\Sigma_{2R}} \left(  |\nabla^T s|^2 +  E
 \right) \notag \\
 & \leq \left( n^2 \, 2^{2\, n -1}+n \right)  \, V(R) + \left(  (n-1) \, 2^{n-1} + 1 \right) \, \int_{\Sigma_{R,2R}} 
	  E \, , \label{e:asin113}
\end{align}
where the last inequality also used that $E \leq  \sum_{j=1}^{n+k} |\nabla^T x_{j}|^2  = n$. 
Since 
Lemma \ref{l:partway}
gives that
\begin{align}
	 \int_{\Sigma_{R,2\, R}}  E &\leq
		  2^{n+1} \, V(R)
	    \, ,
\end{align}
we conclude that
\begin{align}
	V(2R) \leq \left( n^2 \, 2^{2\, n -1}+n \right)  \, V(R) + \left(  (n-1) \, 2^{n-1} + 1 \right) \, 2^{n+1} \, V(R) \leq n^2 \, 2^{2n} \, V(R) \, .
\notag
\end{align}
\end{proof}

\begin{proof}[Proof of Theorem \ref{t:loosedoubling}]
The theorem follows from applying Theorem \ref{t:doubling} to $B_{4r} \cap \Sigma$ twice, first to bound $V(r)$ by $V(r/2)$
and then to bound $V(2r)$ by $V(r)$.  This gives the claim since $V(2r)$ is at least the area in $B_{2r}$, while
the area in $B_r$ is greater than $V(r/2)$.
\end{proof}

\section{Euclidean volume growth}	\label{s:S2}

Throughout this section, $\Sigma^n \subset \RR^{n+k}$ is a proper stationary integral varifold. 
The next result shows that if $\Sigma$   has sublinearly growing height, then it must have Euclidean volume growth.  
Given $a<b$, let $\Sigma_{a,b}$ denote $\{ a < s  < b \} \cap \Sigma$.

\begin{Thm}	\label{t:EVG}
There exists $C=C(n)$, so that for any $\alpha \in [0,1)$ there exists   $\delta = \delta (\alpha,n) > 0$ so that if   $R> 1$, 
$  \Sigma \subset  \{ \frac{1}{4} < s < 4 \, R\}$ is a proper stationary integral varifold, and 
\begin{align}	\label{e:2p2}
	 \Sigma  \subset \{ |\Pi_2 (x)| \leq \delta  \, s^{\alpha} \} \, , 
\end{align}
then $V(R) \leq C \, R^n \, V(1)$.
\end{Thm}

\vskip1mm
We will use the following variation on Lemma \ref{l:tiltgives} in the proof:

\begin{Lem}	\label{l:tiltgives2}
If $\Sigma^n \subset \RR^{n+k}$ is as in Theorem \ref{t:EVG} and $1\leq r \leq 2\, R$, then 
\begin{align}
	\int_{\Sigma_{1,r}} s^{1-n} \,  |\nabla^T s|^2 \leq 2^n\, n^2 \, r \, V(1) +  (n-1) \, r \, \int_{\Sigma_{1,r}} 
	 \sum_{j=1}^k \frac{|\nabla^T x_{n+j}|^2}{s^n}
	\, .
\end{align}
\end{Lem}

\begin{proof}
Let $\phi (s)$ be a piecewise linear function that is $0$ for $s\leq \frac{1}{2}$, has $\phi (1) = 1$, and is zero for $r \leq s$. 
In particular $\phi'(s) = 2$ on $[1/2, 1)$ and $\phi'(s) = - \frac{1}{r-1}$ on $(1, r]$. 

Applying the divergence theorem to $\dv_{\Sigma} \, (\phi (s) \, \nabla s^{2-n})$ gives 
\begin{align}
	0 &= \int_{\Sigma} \phi'(s) \, \langle \nabla^T s , \nabla s^{2-n} \rangle + \phi (s) \, \dv_{\Sigma} (\nabla s^{2-n}) \notag \\
	&= 2 \, (2-n) \,  \int_{\Sigma_{\frac{1}{2},1}} s^{1-n}  \,| \nabla^T s |^2 
	-\frac{2-n}{r-1} \,  \int_{\Sigma_{1,r}} s^{1-n}  \,| \nabla^T s |^2
	 + \int_{\Sigma} \phi (s) \, \dv_{\Sigma} (\nabla s^{2-n}) \, . \notag
\end{align}
Combining this and 
Lemma \ref{l:sdoesk} gives
\begin{align}
	  \int_{\Sigma_{1,r}} s^{1-n}  \,| \nabla^T s |^2 &\leq 2\, r \,  \int_{\Sigma_{\frac{1}{2},1}} s^{1-n}  \,| \nabla^T s |^2 
	  + (n-1) \, r \,  \int_{\Sigma_{ \frac{1}{2} , r} } s^{-n} \, \sum_{j=1}^k |\nabla^T x_{n+j}|^2  \notag \\
	  &\leq 2^n \, r  \, V(1)  
	  + (n-1) \, r \, \left(  n \, 2^n \, V(1) +  \int_{\Sigma_{ 1 , r} }  s^{-n} \sum_{j=1}^k |\nabla^T x_{n+j}|^2 \right) \, . \notag
\end{align}
Simplifying this gives the claim.
\end{proof}

\begin{Lem}	\label{l:usetilt}
If $\Sigma^n \subset \RR^{n+k}$ is as in Theorem \ref{t:EVG} and $ r \leq 2\, R$, then 
\begin{align}
	V(r) \leq \left( 1 + 2^n\, n^2  \, r^n \right) \, V(1) + n \, r^n \, \int_{\Sigma_{1,r}} 
	 \sum_{j=1}^k \frac{|\nabla^T x_{n+j}|^2}{s^n} 
	\, .
\end{align}
\end{Lem}

\begin{proof}
The first claim in Lemma \ref{l:sdoesk}
gives that
	$|\nabla^{\perp} s|^2 
	 \leq  
	\sum_{j=1}^k   | \nabla^T x_{n+j}|^2$.  Using this, $|\nabla^{\perp} s|^2 + |\nabla^{T} s|^2 = 1$, and Lemma 
	\ref{l:tiltgives2} gives that
\begin{align}	\label{e:getcontrol}
	  \int_{\Sigma_{1,r}} s^{1-n} &\leq  \int_{\Sigma_{1,r}} s^{1-n} \, \left( |\nabla^{\perp} s|^2 + |\nabla^{T} s|^2  \right) \leq 
	 \int_{\Sigma_{1,r}} s^{1-n} \,  \left(  |\nabla^{T} s|^2 + \sum_{j=1}^k |\nabla^T x_{n+j}|^2 \right) \notag \\
	&\leq 2^n\, n^2  \, r \, V(1) +  n \, r \, \int_{\Sigma_{1,r}} s^{-n} \, 
	 \sum_{j=1}^k |\nabla^T x_{n+j}|^2	\, .
\end{align}
On the other hand, we have that
\begin{align}
	V(r) \leq V(1) + \int_{\Sigma_{1,r}} \left( \frac{s}{r} \right)^{1-n} = V(1) + r^{n-1} \, 
	\int_{\Sigma_{1,r}} s^{1-n}  \, .
\end{align}
Combining this with \eqr{e:getcontrol} gives the claim.
\end{proof}

The next lemma is a form of monotonicity that allows for boundary or singularities inside of an inner ball $B_{\frac{1}{4}}$.

\begin{Lem}	\label{l:amono}
Suppose that $\Sigma^n$ is a stationary integral varifold in $B_{\bar{r}} \setminus B_{ \frac{1}{4}} \subset \RR^{n+k}$. 
If $1 \leq r_1 \leq r_2 < \bar{r}$, then 
\begin{align}
	r_1^{-n} \,  \Vol \left( B_{r_1} \cap \Sigma \setminus B_{ \frac{1}{2}} \right) \leq r_2^{-n} \,  \Vol ( B_{r_2} \cap \Sigma) \, .
\end{align}
\end{Lem}

\begin{proof}
Let $\phi (s) $ be a Lipschitz function that is one for $s \leq \frac{1}{2}$, zero for $s \geq 1$, and is non-increasing.  Given $r \geq 1$, define $\gamma = \gamma^r$ by 
\begin{align}
	\gamma (s) = 
	\begin{cases}
	0 & {\text{ if }} s \leq \frac{1}{4} \, , \\
	4\,s - 1 & {\text{ if }}  \frac{1}{4} \leq s \leq \frac{1}{2} \, , \\
	\phi (s/r) & {\text{ if }}  \frac{1}{2} \leq s \, . 
	\end{cases}
\end{align}
Given $r \in [1, \bar{r}]$, define $I(r)$ by 
\begin{align}
	I (r) = \int_{\Sigma} \gamma (|x|) \, .
\end{align}
Differentiating $I(r)$ with respect to $r$  gives
\begin{align}	\label{e:Ipr}
	I'(r) =  -  \int_{\Sigma} \phi' (|x|/r) \, \frac{|x|}{r^2} \, .
\end{align}
Applying the divergence theorem to the compactly supported vector field $\gamma(|x|) \, x$ gives
\begin{align}
	0 &= \int_{\Sigma} \dv_{\Sigma} (\gamma (|x|)  \, x)  = n\, \int_{\Sigma} \gamma (|x|)  +  \int_{\Sigma}  \gamma' (|x|)  \, \langle \nabla^T |x| , x \rangle  \notag \\
	&= n \, I(r) + 4\, \int_{B_{\frac{1}{2}} \cap \Sigma \setminus B_{\frac{1}{4}}} \frac{| x^T|^2}{|x|} + \frac{1}{r} \,  \int_{\Sigma} \phi' (|x|/r) \, \frac{|x^T|^2}{|x|} \\
	 &= n \, I(r) + 4\, \int_{B_{\frac{1}{2}} \cap \Sigma \setminus B_{\frac{1}{4}}} \frac{| x^T|^2}{|x|} 
	 	- r\, I'(r)  -  \frac{1}{r} \,  \int_{\Sigma} \phi' (|x|/r) \, \frac{|x^{\perp}|^2}{|x|}
\, , \notag
\end{align}
where the last equality used \eqr{e:Ipr}. Since $\phi' \leq 0$, we conclude that
\begin{align}
	n \, I(r) - r \, I'(r) \leq 0 \, .
\end{align}
Integrating this differential inequality from $r_1$ to $r_2$ gives that
\begin{align}
	\frac{I(r_1)}{r_1^n} \leq \frac{I (r_2)}{r_2^n} \leq \frac{\Vol (B_{r_2} \cap \Sigma)}{r_2^n} \, .
\end{align}
Taking a sequence of $\phi$'s that go to one on $[0,1]$, the monotone convergence theorem gives 
\begin{align}
	\Vol \left(B_{r_1} \cap \Sigma \setminus B_{ \frac{1}{2}} \right) \leq I(r_1) \leq r_1^n \, r_2^{-n} \,  \Vol (B_{r_2} \cap \Sigma)  \, . 
\end{align}
\end{proof}

 \vskip1mm
The next lemma is a reverse Poincar\'e inequality:

\begin{Lem}	\label{c:RP}
If   
 $\Sigma^n \subset \RR^{n+k}$ is a stationary integral varifold in $\{ \frac{r_1}{2} < s < 2\, r_2 \} \subset \RR^{n+k}$,
 then 
\begin{align}
	\int_{ \Sigma_{r_1 , r_2} }  \sum_{j=1}^k|\nabla^T x_{n+j}|^2 
	\leq  16 \, r_1^{-2} \, \int_{ \Sigma_{ \frac{r_1}{2}, r_1}} |\Pi_2 (x)|^2  + 4 \, r_2^{-2} \,
	\int_{ \Sigma_{r_2, 2\, r_2}}  |\Pi_2 (x)|^2
	 \, .
\end{align}
\end{Lem}

\begin{proof}
 Let $\phi$ be a cutoff depending on $s$ so that $\phi \equiv 1$ on $r_1< s < r_2$, $\phi$ is zero outside of $\frac{r_1}{2} \leq s \leq 2\,  r_2$, and $\phi$ is linear in between. 
Since $\dv_{\Sigma} \nabla  x_{n+j} = 0$, the divergence theorem gives for each $j$ that 
\begin{align}
	0 &= \int \dv_{\Sigma} (\phi^2 \, x_{n+j} \, \nabla x_{n+j}) = \int \left(   \phi^2 \, |\nabla^T x_{n+j}|^2 +2\, \phi \, x_{n+j} \, \langle \nabla^T \phi , \nabla x_{n+j} \rangle
	\right) \notag \\
	&\geq \int \left(  \frac{1}{2} \,  \phi^2 \, |\nabla^T x_{n+j}|^2 -2\,  x_{n+j}^2 \,  | \nabla^T \phi|^2
	\right) \, , \label{e:RPhere}
\end{align}
where the inequality used the absorbing inequality.  Thus, we get that
\begin{align}
	\int_{ \Sigma_{r_1 , r_2} }  |\nabla^T x_{n+j}|^2 \leq 4 \, \int_{\Sigma} x_{n+j}^2 \, |\nabla^T \phi|^2 
	\leq  16 \, r_1^{-2} \, \int_{ \Sigma_{ \frac{r_1}{2}, r_1}} x_{n+j}^2  + 4 \, r_2^{-2} \,
	\int_{ \Sigma_{r_2, 2\, r_2}} x_{n+j}^2 
	 \, .
\end{align}
The lemma follows from summing  this over $j= 1, \dots , k$.
\end{proof}

\vskip1mm

\begin{proof}[Proof of Theorem \ref{t:EVG}]
The constant $\delta > 0$ will be chosen at the end of the proof; for now, we will just assume that $\delta$ is small enough, depending just on $n$, so
 that
Corollary \ref{c:doubling} gives that $V(2\, R) \leq C_n \, V(R)$ with $C_n$ depending only on $n$.

\vskip2mm
\noindent
{\bf{Reduction to $R=2^N$}}. We first observe that it suffices to prove the bound when $R=2^N$ for some natural number $N$.  To see this, suppose that we have proven the bound whenever $R= 2^m$ and suppose that 
some $R$ satisfies
$2^{m-1} < R < 2^{m}$.  Then we have
\begin{align}
	V(R) \leq V( 2^m) \leq C \, (2^m)^n \, V(1) \leq C \, (2\, R)^n \, V(1) = 2^n \, C \, R^n \, V(1) \, ,
\end{align}
  giving   the  general case.

\vskip2mm
\noindent
{\bf{The proof when $R = 2^N$}}. 
 Define $r_i$ and  $v_i > 0$ by 
 \begin{align}
 	r_i = 2^i {\text{ and }} v_i = \Vol (\Sigma_{r_i , r_{i+1}}) \, .
\end{align}
  We claim that 
  	for $i \leq N$
\begin{align}	\label{e:boundVi}
	r_{i+1}^{-n} \, v_i \leq  V(1) + R^{-n} \, V (2\, R) \, .
\end{align}
This is automatic for $i=-1$ or $i=N$.  
Using the slow growth of the height, it follows that $\Sigma_{r_i , r_{i+1}} \subset B_{ \sqrt{1+ \delta^2} \, r_{i+1}} \setminus B_{r_i}$.
Therefore  Lemma \ref{l:amono} gives for each $i \in \{ 0 , \dots , N-1\}$ that 
\begin{align}
	v_i &\leq \Vol \left( B_{ \sqrt{1+ \delta^2} \, r_{i+1}} \cap \Sigma \setminus B_1 \right) 
	\leq  (1+ \delta^2)^{ \frac{n}{2}}  \, r_{i+1}^n \, (2\, R)^{-n} \, \Vol (B_{2R} \cap \Sigma) \notag \\
	&\leq  2^{-n} \, (1+ \delta^2)^{ \frac{n}{2}}  \, r_{i+1}^n \, R^{-n} \, V(2\, R) \, .
\end{align}
This gives \eqr{e:boundVi} as long as we take $\delta^2 < 3$.

 To shorten notation, define $E= E(x)$ to be the tilt-excess at each point
 \begin{align}
 	E =   \sum_{j=1}^k |\nabla^T x_{n+j}|^2 \, .
 \end{align}
A dyadic decomposition gives that
 \begin{align}	\label{e:dya}
 	\int_{\Sigma_{1,R}} 
	s^{-n} \, E  = \sum_{i=0}^{N-1}  \int_{\Sigma_{r_{i},r_{i+1}}} 
	s^{-n} \, E \leq 
	 \sum_{i=0}^{N-1}  r_i^{-n} \, \int_{\Sigma_{r_{i},r_{i+1}}} 
	E \, .
 \end{align}
Lemma \ref {c:RP}
 gives for each $i$ that 
\begin{align}
		\int_{\Sigma_{2^{i},2^{i+1}}} 
	 E &\leq
	  16 \, r_i^{-2} \, \int_{ \Sigma_{ r_{i-1}, r_i}} |\Pi_2 (x)|^2   + 4 \, r_{i+1}^{-2} \,
	\int_{ \Sigma_{r_{i+1}, r_{i+2}}} |\Pi_2 (x)|^2 \notag \\
	&\leq 2^{4-2\, i} \,  v_{i-1} \, \delta^2 \, 2^{2 \, i \, \alpha}  + 2^{-2\, i} \,
	 v_{i+1} \, \delta^2 \, 2^{2\, \alpha \, (i+2)}   	 \, .
\end{align}
Simplifying and using this in \eqr{e:dya} gives
 \begin{align}	\label{e:dya2}
 	\int_{\Sigma_{1,R}} 
	s^{-n} \, E & \leq 16 \, \delta^2 \, 
	    \sum_{i=0}^{N-1}   r_i^{2\, \alpha - 2} \, \left( \frac{ v_{i-1}}{r_i^n} +   2^{2\, n} \, 
	 \, \frac{ v_{i+1}}{r_{i+2}^n} \right)   \notag \\
	 &\leq   16 \, \delta^2 \, \left( 1 + 2^{2\, n} \right) \, \left( V(1) + R^{-n} \, V (2\, R) 
	 \right) 
	    \sum_{i=0}^{N-1}   r_i^{2\, \alpha - 2} \\
	    &\leq    16\,  \delta^2 \, C_{\alpha} \,  \left( 1 + 2^{2\, n} \right) \, \left( V(1) + R^{-n} \, V (2\, R) 
	 \right) 
	  \, , \notag
 \end{align}
 where the  second to last inequality used \eqr{e:boundVi} and the last inequality used that $\alpha < 1$ so the geometric series is summable (independent of $N$).
 Using this in Lemma \ref{l:usetilt}
 gives
 \begin{align}	 
	V(R) &\leq \left( 1 + 2^n\, n^2  \, R^n \right) \, V(1) + n \, R^n \, C_{\alpha} \, \delta^2 \,  \left( 1 + 2^{2\, n} \right) \, \left( V(1) + R^{-n} \, V (2\, R) 
	 \right) 
 \notag \\
	& \leq 
	 (C_n' + C_{\alpha} \, \delta^2)  \, R^n  \, V(1) + C_n' \,   C_{\alpha} \, \delta^2 \, V (2\,R)  \\
	 &\leq 
	 	 (C_n' + C_{\alpha} \, \delta^2)  \, R^n  \, V(1) + C_n' \,   C_n \, C_{\alpha} \, \delta^2 \, V (R)  
	\, ,  \notag
\end{align}
where the last inequality used the doubling bound from Corollary \ref{c:doubling}  and the constants $C_n'$ and 
$C_n$ depend only on $n$. 
Choose $\delta > 0$ small depending on $n$ and $\alpha$, so we can absorb the last term on the right to complete the proof.
\end{proof}

\section{Bernstein theorems}	\label{s:S4}

We are now ready to prove the Bernstein theorem for proper stable minimal hypersurfaces with sublinearly growing height in any dimension.  The height bound is used to prove Euclidean volume growth and to get  decay of the tilt-excess
\begin{align}
	R^{-n} \, \int_{ \{ s \leq R \} } |\nabla^T x_{n+1}|^2 \, .
\end{align}
The tilt-excess may be much larger than the average of $ |\nabla^T x_{n+1}|^2$ since we divide only by $R^n$ rather than the volume.  This is why we need the area bound in order to show that the scale-invariant tilt-excess is small enough
to apply
 the sheeting theorem of \cite{Be,SS,dLHS}; cf. \cite{CaP} for stability results in a related direction.

\begin{proof}[Proof of Theorem \ref{t:SSB}]
Since $\Sigma$ has sublinearly growing height, 
Theorem \ref{t:EVG} gives that $\Sigma$ has Euclidean volume growth
\begin{align}
	\Vol (B_R \cap \Sigma) \leq C \, R^n \, , 
\end{align}
where $C$ depends on the volume of $\Sigma$ in a fixed compact set and the rate of sublinear growth of the height. 

After rotating, translating and rescaling, we can assume that $\Sigma$ is contained in the conical region $|x_{n+1}| \leq c \, |x|^{\alpha}$,
  for all $|x| \geq 1$, where $\alpha < 1$ and $c$ is fixed.
Since the function $x_{n+1}$ is harmonic on $\Sigma$, 
the reverse Poincar\'e  gives  for all $R > 4$ that
\begin{align}	\label{t:RPbern}
	\int_{\{ s \leq R\}  \cap \Sigma} |\nabla^T x_{n+1}|^2 & \leq \frac{4}{R^2} \, \int_{\{ s \leq 2R\} \cap \Sigma} x_{n+1}^2 
	\leq \frac{4\, c^2 \, (3R)^{2\, \alpha}}{R^2} \,  \Vol  ( \{ s \leq 2R\} \cap \Sigma) \notag \\
	&\leq 4 \, c^2\, 3^{2\alpha}\,  R^{2\, \alpha - 2} \,  \Vol  (B_{3R} \cap \Sigma) 
	\leq  4\, c^2 \, C \,  3^{n+ 2\alpha} \, R^{n+ 2\, \alpha -2} \, .
\end{align}
 Theorem $4$ in \cite{Be} gives   constants $\kappa > 0$ and $C_B$ depending on $n$ so that if 
\begin{align}	\label{e:Belle}
	R^{-n} \,  \int_{\{ s \leq R\}  \cap \Sigma} |\nabla^T x_{n+1}|^2 \leq \kappa   \, , 
\end{align}
then 
\begin{align}	\label{e:Bell2}
	\sup_{ \{ s \leq R/2\} \cap \Sigma} \, |\nabla^T x_{n+1}|^2 \leq  C_B \, R^{-n} \,  \int_{\{ s \leq R\}  \cap \Sigma} |\nabla^T x_{n+1}|^2 \, .
\end{align}
Since $\alpha < 1$,  \eqr{t:RPbern}  implies that \eqr{e:Belle} holds for every $R$ large enough.   It follows that \eqr{e:Bell2} holds for all large $R$ and
\begin{align}	\label{e:Bell23}
	\sup_{ \{ s \leq R/2\} \cap \Sigma} \, |\nabla^T x_{n+1}|^2 \leq C' \, R^{2\, \alpha - 2}   \, .
\end{align}
Taking $R \to \infty$  implies that 
$x_{n+1}$ is constant on $\Sigma$.
\end{proof}

\subsection{Stable hypersurfaces with boundary}

Stable minimal hypersurfaces are often produced by solving a Plateau problem, including where there is 
a fixed interior boundary, cf. \cite{CM2,CM3,CM4,CM5,CM6}. This makes it useful to have 
  an asymptotic description when the stable hypersurface has compact boundary:

\begin{Thm}	\label{t:SSB2}
Suppose that $\Sigma^n \subset \RR^{n+1}$ is a complete properly embedded stable minimal hypersurface with compact boundary $\partial \Sigma$.  If $\Sigma$ is two-sided and has height growing at the rate $\alpha < 1$,
then outside of a compact set $|A|^2 \leq C \, |x|^{2\, \alpha -4}$ and $\Sigma$ consists of a collection of disjoint graphs of functions $u_1 , \dots , u_k$ with $\nabla u_i \to 0$ at $\infty$ and each $u_i$ goes to a constant.
\end{Thm}

\begin{proof} 
After scaling, we can assume that $\partial \Sigma \subset B_{ \frac{1}{4}}$.
Theorem \ref{t:EVG} gives that  $\Sigma$ has  Euclidean volume growth.  Lemma \ref{c:RP}
gives for all $R$ large enough that
\begin{align}
	R^{-n} \, \int_{\{ R \leq s \leq 2\, R\}  \cap \Sigma} |\nabla^T x_{n+1}|^2 & \leq C\, R^{ 2\, \alpha -2} \, , 
\end{align}
where $\alpha < 1$, $C$ is independent of $R$, and $C$ depends on $\alpha$, $n$, and 
 the volume of $\Sigma$ in a fixed compact set.  For each $x_0\in \Sigma$ with $|x_0|$ large, we can
apply Theorem $4$ in \cite{Be} on $B_{\frac{|x_0|}{2}}(x_0)$ to get that $\Sigma$ is graphical{\footnote{In the case where $\Sigma$ is a stable stationary integral varifold whose singular set has locally finite codimension two measure, we would now get  full regularity of $\Sigma$ outside of a fixed compact set. The rest of the argument below would apply without change.}} over $B_{\frac{|x_0|}{4}}(x_0)$ with  gradient at most $C \, |x_0|^{\alpha -1}$.
Combining this with embeddedness   gives that $\Sigma$ is a union of graphs of functions $u_i$ over the complement of a
 ball with $|\nabla u_i| \to 0$ at infinity.  Proposition 3 in \cite{Sc} (cf. \cite{LSW}) shows that each $u_i$ converges to a constant height at infinity.

Theorem $1$ in \cite{SS}   gives that $|A|^2 \leq C \, |x|^{-2}$.  
Finally, we will improve the $|A|^2 \leq C \, |x|^{-2}$ bound to get $|x|^{2\, \alpha -4}$ decay.   Fix a large $x_0$, set $R = |x_0|/2$, and let  $\phi$ be a cutoff function that is one on $B_{R/2} (x_0)$ and zero outside of $B_R (x_0)$.  Applying the divergence theorem to $\phi^2 \, \nabla^T  |\nabla^T x_{n+1}|^2$ gives
\begin{align}
 0 =  \int \left(
 4 \, \phi \, |\nabla^T x_{n+1}| \langle \nabla \phi , \nabla^T |\nabla^T x_{n+1}| \rangle + 2\,  \phi^2 \, [ |\Hess^T_{x_{n+1}}|^2 + \Ric (\nabla^T x_{n+1} , \nabla^T x_{n+1})]
 \right) \, .
\end{align}
Using the bounds $|\Ric| \leq c \, |A|^2 \leq c \, R^{-2}$ and $|\nabla^T x_{n+1}| \leq c \, R^{\alpha -1}$ 
on the support of $\phi$ and using an absorbing inequality and the Euclidean volume growth, we see that
\begin{align}	\label{e:hessTb}
	\int_{B_{R/2} (x_0)}  |\Hess^T_{x_{n+1}}|^2 \leq  C \, R^{n+ 2\, \alpha -4} \, .
\end{align}
Note that if $e_i$ is a frame for $\Sigma$, then
\begin{align}
	\Hess_{x_{n+1}}^T (e_i ,e_j) = \langle \nabla_{e_i} \partial_{n+1}^T , e_j \rangle = -  \langle \nabla_{e_i} \partial_{n+1}^{\perp} , e_j \rangle
	=  \langle \partial_{n+1}^{\perp} ,  \nabla_{e_i} e_j \rangle = A_{ij} \, \langle \partial_{n+1} , \nn \rangle \, , 
\end{align}
so that 
\begin{align}
	| \Hess_{x_{n+1}}^T |^2 = |A|^2 \, \langle \partial_{n+1} , \nn \rangle^2 = |A|^2 \, (1 - |\nabla^T  x_{n+1} |^2) \geq \frac{1}{2} \,  |A|^2 \, , 
\end{align}
where the last inequality used the graph condition.   Therefore, \eqr{e:hessTb} gives 
\begin{align}	\label{e:hessTb2}
	\int_{B_{R/2} (x_0)}  |A|^2 \leq  C \, R^{n+ 2\, \alpha -4} \, .
\end{align}
Since $\Delta |A|^2 \geq - 2\,  |A|^4 \geq - C \, R^{-2} \, |A|^2$ by the Simons inequality (see, e.g., Lemma $2.1$ in \cite{CM1}), the meanvalue inequality (Corollary $1.16$ in \cite{CM1}) applies on the   scale $R$  to give
\begin{align}
	|A|^2 (x_0) \leq C \, R^{2\, \alpha -4} \, ,
\end{align}
completing the proof.
\end{proof}

\subsection{Curvature estimate in all dimensions}

The next theorem gives curvature estimates for compact stable minimal hypersurfaces that extend sublinearly past the given ball.  This curvature estimate holds in all dimensions.

\begin{Thm}	\label{t:LCE}
Given $V>0$, $1>\alpha$ and $n$, there exist $C=C(V,\alpha,n)>1$ and $\delta = \delta (n) > 0$  such that if $\Sigma\subset B_{C\,r}\subset \RR^{n+1}$ is a compact stable two-sided minimal hypersurface with
 $\partial \Sigma\subset \partial B_{C\,r}$, $\Vol (B_r\cap\Sigma)\leq V\,r^n$ and $\Sigma \setminus B_r \subset \{|x_{n+1}|< \delta \, r^{1-\alpha}\,|x|^{\alpha}\}$, then for $s = \frac{C\,r}{8}$
\begin{align}
\sup_{B_{ s} \cap \Sigma}|A|^2\leq s^{-2} \,  .
\end{align}
\end{Thm}

\begin{proof}
We roughly follow the argument for  Theorem \ref{t:SSB}. First,   
Theorem \ref{t:EVG} gives $c_1$ and $c_2$ so  that  for $R \leq C \, r/4$
\begin{align}	\label{e:GTSG}
	\Vol (B_R \cap \Sigma) \leq c_1\, \Vol (B_{c_2 \, r} \cap \Sigma)  \, \left( \frac{R}{c_2 \, r} \right)^n \, , 
\end{align}
where $c_1$ and $c_2$ depend on  $n$ and $\alpha$ but do not depend on $C$.   Theorem \ref{t:doubling} gives $\delta = \delta (n)$ 
and $c_3 = c_3 (n)$ so that
\begin{align}	\label{e:KTSG}
	\Vol (B_{c_2 \, r} \cap \Sigma)  \leq  c_2^{c_3} \, \Vol (B_r \cap \Sigma) \leq c_2^{c_3} \,  V \, .
\end{align}
Set $R= C \, r / 4$. 
Combining  \eqr{e:GTSG} and \eqr{e:KTSG} gives 
\begin{align}	\label{e:GTSG2}
	\Vol (B_R \cap \Sigma) \leq c_4 \, V  \, R^n \, , 
\end{align}
where $c_4$ depends on $n$ and $\alpha$. 
Using this, the
reverse Poincar\'e inequality gives that
\begin{align}
	\int_{B_{R/2} \cap \Sigma} |\nabla^T x_{n+1}|^2 &\leq \frac{16}{R^2}  \, \left( r^{1-\alpha} \, R^{\alpha} \right)^2 \, \Vol (B_R \cap \Sigma)
	\leq 16\, c_4 \, V \, R^{n}  \, \left( \frac{r}{R} \right)^{2-2\, \alpha} \notag \\ 
	&= 16\, c_4 \, V \, R^{n} \, \left( \frac{4}{C} \right)^{2-2\, \alpha} 
	\, .
\end{align}
Since $\alpha < 1$,  we can choose $C$ sufficiently large (in terms of $n, \alpha$ and $V$) so that the right-hand side is 
less than $\kappa \, (R/2)^n$ where $\kappa > 0$ (depending on $n$) is given by Theorem $4$ in \cite{Be}. We conclude that 
 each component of $B_{R/4}\cap \Sigma$ is a graph with   gradient bounded by at most $1$.
   Finally, the desired curvature estimate follows from Theorem $1$ in \cite{SS}.
 \end{proof}

\vskip1mm
We would like to thank Claudio Arezzo for raising this question.

\section{Rate of convergence at infinity}	\label{s:S3}

The main result in this section (Theorem \ref{t:rate}) gives a rate of convergence for the volume ratio of a stationary integral varifold that is contained in  a slab.
The key is a weighted integrability of the tilt-excess.    Theorem  \ref{t:EVG} and 
 Lemma \ref{c:RP}
give for each $R$  that 
\begin{align}
	\sum_{j=1}^k \int_{B_R \cap \Sigma } |\nabla^T x_{n+j}|^2\leq C \, R^{n-2} \,  ,
\end{align}
which implies that $\sum_{j=1}^k \int_{B_R \cap \Sigma } |\nabla^T x_{n+j}|^2\, |x|^{2-n} \leq C' \log R$.  
We need a better estimate to get the sharp rate of convergence.  The next theorem achieves this
by proving  integrability:

\begin{Thm}	\label{t:slab2}
There exist $C$ and $R_0$ depending only on $n$, so that if
 $\Sigma^n \subset \RR^{n+k}$ is a complete  proper stationary integral varifold in the slab $|\Pi_2 (x)| \leq 1$, then
\begin{align}
	\sum_{j=1}^k \int_{\Sigma}|\nabla^Tx_{n+j}|^2\,|x|^{2-n} \leq C \, V(R_0) \,  .
\end{align}
\end{Thm}

\begin{proof} 
As before, let $E=E(x) =  \sum_{j=1}^k  |\nabla^T x_{n+j}|^2 $ denote the tilt-excess at each point.
Since $\Sigma$ is in a slab, 
Theorem \ref{t:EVG} gives $R_0$ and $C$ depending only on $n$ so that  for $R \geq R_0$
 \begin{align}
 	\Vol (B_R \cap \Sigma) \leq V(R) \leq   C\, V(R_0) \, \left( \frac{R}{R_0} \right)^n \, .
\end{align}
Set  $\cV = C\, R_0^{-n} \, V(R_0)$,   so that  $\Vol (B_R \cap \Sigma)  \leq \cV \, R^n$. 
  The meanvalue inequality for stationary varifolds (see, e.g., Proposition $3.8$ in 
\cite{CM1}) applied to $x_{n+j}^2$ gives that
\begin{align}
	   R^{-n}\,  \int_{B_R \cap \Sigma} x_{n+j}^2 - R_0^{-n} \, \int_{B_{R_0} \cap \Sigma} x_{n+j}^2 =
	   	 \int_{B_R \cap \Sigma \setminus B_{R_0}} \frac{ x_{n+j}^2 \, |x^{\perp}|^2}{|x|^{n+2}}
	+ \int_{R_0}^R \frac{1}{r^{n+1}} \, \int_{B_r \cap \Sigma} (r^2 - |x|^2) \, |\nabla^T x_{n+j}|^2 \notag  \, .
\end{align}
Since the first term  on the right is nonnegative, summing this over $j= 1, \dots , k$ gives
\begin{align}
	\frac{3}{4} \,   \int_{R_0}^R r^{1-n} \,  \int_{B_{r/2} \cap \Sigma}   \, E  &\leq 
	\int_{R_0}^R r^{-n-1} \,  \int_{B_r \cap \Sigma} (r^2 - |x|^2) \, E 
	 \leq R^{-n}\,  \int_{B_R \cap \Sigma} |\Pi_2 (x)|^2 \leq \cV \, . \label{e:RfH}
\end{align}
Define   $F(r)$ by $F(r) =    \int_{B_r \cap \Sigma}   E$.  
Since \eqr{e:RfH} holds for every $R$, it follows that
\begin{align}	 
	  \int_{R_0}^{\infty} r^{1-n} \, F(r/2)   \, dr  \leq \frac{ 4}{3} \,  \cV \, .
\end{align}
Using the change of variables $\tau = \frac{r}{2}$, 
this gives that
\begin{align}	\label{e:222}
	  \int_{\frac{R_0}{2}}^{\infty} \tau^{1-n} \, F(\tau)   \, d\tau =  2^{n-2} \,   \int_{R_0}^{\infty} r^{1-n} \, F(r/2)   \, dr
	    \leq \frac{2^n}{3} \,  \cV \, .
\end{align}
 On each dyadic interval $[2^i \, R_0 , 2^{i+1} \, R_0]$, we have that
\begin{align}
	\int_{2^i\, R_0}^{2^{i+1} \, R_0} r^{1-n} \, F(r) &\geq F(2^i\, R_0) \, \int_{2^i\, R_0}^{2^{i+1} \, R_0} r^{1-n}
	 = \frac{F(2^i \, R_0)}{n-2} \,  \left( (2^i \, R_0)^{2-n} - (2^{i+1}\, R_0)^{2-n} \right) \notag \\
	 &\geq \frac{F(2^i\, R_0)}{2(n-2)} \, 2^{i(2-n)} \, R_0^{2-n}\, .
	 \notag
\end{align}
Thus, we see that
\begin{align}
	\int_{B_{ 2^{i+1} \, R_0 }\setminus B_{2^i \, R_0}} 
	 |x|^{2-n}\, E &\leq \left( 2^{i} \, R_0 \right)^{2-n} \, 	\int_{B_{ 2^{i+1} \, R_0 }\setminus B_{2^i \, R_0}} E
	\leq \left( 2^{i} \, R_0 \right)^{2-n} \, F(2^{i+1} \, R_0) \notag \\
	&\leq \left( 2^{i} \, R_0 \right)^{2-n} \, (2n-4) \, R_0^{n-2} \, 2^{(i+1) \, (n-2)} \, \int_{2^{i+1}\, R_0}^{2^{i+2} \, R_0} r^{1-n} \, F(r) \\
	&=    (n-2)  \, 2^{ n-1} \, \int_{2^{i+1}\, R_0}^{2^{i+2} \, R_0} r^{1-n} \, F(r) 
	\, . \notag
\end{align}
Summing this over $i$ and using \eqr{e:222} gives that
\begin{align}
	\int_{\Sigma \setminus B_{R_0}} |x|^{2-n} \, E \leq  (n-2)  \, 2^{ n-1} \, \int_{2R_0}^{\infty} r^{1-n} \, F(r) \leq 
	 \frac{(n-2)}{3}  \, 2^{ 2\,n-1} \, \cV \, .
\end{align}
To bound the integral on $B_{R_0}$, we use monotonicity and a dyadic decomposition to get
\begin{align}
	\int_{B_{R_0} \cap \Sigma} |x|^{2-n} &= \sum_{i=0}^{\infty} \int_{B_{2^{-i} \, R_0} \cap \Sigma \setminus B_{2^{-i-1}\, R_0} } |x|^{2-n}
	\leq \sum_{i=0}^{\infty} \left( 2^{-i-1}\, R_0 \right)^{2-n} \, \Vol (B_{2^{-i} \, R_0} \cap \Sigma) \notag \\
	&\leq C \, V(R_0) \, R_0^{2-n}\, 2^{n-2} \,   \sum_{i=0}^{\infty} 2^{-2\, i } \leq C' \, V(R_0) \, .
\end{align}
This bounds the rest of the integral since 
	 $E \leq \sum_{i=1}^{n+k}  |\nabla^T x_{i}|^2 = n$.
\end{proof}

We will  use this  integrability to prove a rate of convergence for the volume ratios.  The following lemma will also be used:

\begin{Lem}	\label{l:pullout}
If $\Sigma^n \subset \RR^{n+k}$ has $|\Pi_2 (x)| \leq 1$ for all $x \in\Sigma$, then
\begin{align}	\label{e:pullout}
	|x^{\perp}|^2  \leq 2 + 2\, |x|^2 \, \sum_{j=1}^k   | \nabla^T x_{n+j}|^2 \, .
\end{align} 
\end{Lem}

\begin{proof}
 Since $x = \Pi_1 (x) + \Pi_2 (x)$ and orthogonal projection is linear, the squared triangle inequality gives 
\begin{align}
	|x^{\perp}|^2 =  |(\Pi_1(x))^{\perp} + (\Pi_2 (x))^{\perp} |^2 \leq 2 \,  |(\Pi_1(x))^{\perp}|^2  + 2\, |(\Pi_2 (x))^{\perp} |^2 \, .
\end{align}
The last term is bounded by $2\, |\Pi_2 (x)|^2 \leq 2$.  To bound the other term, 
   let $\nn_1 , \dots , \nn_k$ be an orthonormal normal frame for $\Sigma$ at $x$ and observe that 
\begin{align}
	2 \,  |(\Pi_1(x))^{\perp}|^2 &\leq 2\,  \sum_{j=1}^k   \langle \Pi_1 (x)    , \nn_j \rangle^2 = 2\,  \sum_{j=1}^k   \langle \Pi_1 (x)    , \Pi_1 (\nn_j) \rangle^2 
	\leq 2 \, |x|^2 \, \sum_{j=1}^k  | \Pi_1 (\nn_j) |^2 \notag \\
	&\leq 2 \, |x|^2 \, \sum_{j=1}^k (1 -  | \Pi_2 (\nn_j) |^2) \leq 2 \, |x|^2 \, \sum_{j=1}^k   | \nabla^T x_{n+j}|^2 \, , 
\end{align}
where the last inequality used the inequalities in the last two lines of 
\eqr{e:nabsk}. Putting this all together gives \eqr{e:pullout}.
\end{proof}

\begin{proof}[Proof of Theorem \ref{t:rate}]
The statement is invariant under translation in the plane $\Pi_2 = 0$, so we can assume that $p=0$. 
Theorem \ref{t:EVG}  gives $R_0$ and $C$ depending only on $n$ so that  for $r \geq R_0$
 \begin{align}
 	\Vol (B_r \cap \Sigma) \leq V(r) \leq   C\, V(R_0) \, \left( \frac{r}{R_0} \right)^n \, .
\end{align}
By monotonicity (see, e.g. Proposition $3.7$ in \cite{CM1}),  
the limiting density 
\begin{align}
	\bV = \lim_{r \to \infty} \frac{ \Vol (B_r \cap \Sigma)}{\Vol (B_r \subset \RR^n)} 
\end{align}
 exists and satisfies $ \bV \leq C \, \frac{V(R_0)}{\omega_n \, R_0^n}$ where 
  $\omega_n = \Vol (B_1 \subset \RR^n)$.  We will later show that $\bV$ is an integer.
  
  Taking the limit as the outer radius goes to infinity in monotonicity gives that
\begin{align}	\label{e:mono1}
	 \bV -  \frac{r^{-n}}{\omega_n} \,  \Vol (B_r \cap \Sigma) =\frac{1}{ \omega_n} \, 
	  \int_{\Sigma \setminus B_r} \frac{|x^{\perp}|^2}{|x|^{n+2}} \, .
\end{align}
Using Lemma \ref{l:pullout} in  \eqr{e:mono1} gives 
\begin{align}	\label{e:mono2}
	 \bV - \frac{\Vol (B_r \cap \Sigma)}{  \omega_n \, r^n}  &\leq \frac{2}{\omega_n}\,  \int_{\Sigma \setminus B_r} \left( |x|^{-n-2} + |x|^{-n} \, \sum_{j=1}^k |\nabla^T x_{n+j}|^2
	\right) \notag \\
	&\leq \frac{2}{\omega_n} \, \int_{\Sigma \setminus B_r} |x|^{-n-2} + \frac{2}{\omega_n \, r^2} \, \sum_{j=1}^k  \int_{\Sigma \setminus B_r} |x|^{2-n} \, |\nabla^T x_{n+j}|^2 \, .
\end{align}
Theorem 	\ref{t:slab2} gives that the last term is bounded by 
\begin{align}
	r^{-2} \, \sum_{j=1}^k  \int_{\Sigma \setminus B_r} |x|^{2-n} \, |\nabla^T x_{n+j}|^2 \leq
 	 C \, V(R_0) \, r^{-2} \leq
	 C \, \bV \, r^{-2}  \,  .
\end{align}
To bound the first term on the last line of \eqr{e:mono2}, use monotonicity to get
\begin{align}
	 \int_{\Sigma \setminus B_r} |x|^{-n-2} &\leq \sum_{j=1}^{\infty} \int_{B_{2^j \, r} \setminus B_{2^{j-1} \, r}} |x|^{-n-2}
	 \leq    \sum_{j=1}^{\infty}  \Vol (B_{2^j \, r} \cap \Sigma) \,  (2^{j-1} \, r)^{-n-2}  \notag \\
	 &\leq  \omega_n \, \bV \, \sum_{j=1}^{\infty}  (2^j \, r)^n  \,  (2^{j-1} \, r)^{-n-2} = 
	 \frac{2^{n+2}}{r^2} \,  \omega_n \, \bV \, \sum_{j=1}^{\infty}    2^{-2j} = 
	  \frac{2^{n+2}  \,  \omega_n \, \bV}{3\, r^2}
	 \, .
\end{align}
  Finally, the integrality of $\bV$ follows from our bounds on the density and excess together with Allard's integrality theorem, \cite{Al} or $42.8$ in \cite{Si}.
\end{proof}

\section{One-sided volume bounds}	\label{s:S5}

We turn now to Theorem \ref{t:boundingoutB} where a minimal submanifold ``lies on one side'' of an $n$-plane.  This theorem bounds the volume of the minimal submanifold
near the $n$-plane in a larger cylinder in terms of the volume in a smaller cylinder.

Throughout this section,   $\Sigma^n \subset \RR^{n+k}$ is minimal with $n \geq 3$, $s = \sqrt{ \sum_{i=1}^n x_i^2 }$ is the distance to the ``vertical axis'', 
and we define vertical cylinders $\cC_r$ by
\begin{align}
	\cC_r = \{ s < r \}  \, .
\end{align}
We also define vertically truncated cylinders $\cC_{r,b} = \cC_r \cap \{ x_{n+j} < b, \, j=1 , \dots , k \}$.

\vskip1mm
Unlike the case where $\Sigma$ is contained in a slab, it is necessary here to decrease the height   to get an area bound in a larger cylinder. 
This can be seen  in examples with tilted planes:

\begin{Exa}
Suppose that $ R \geq 8$ and $\Sigma$ is the graph of $x_{n+1} = 4 - \frac{x_1}{R}$. Then $\cC_{4R} \cap \Sigma \subset \{ x_{n+1} > 0\}$ and 
$\Vol (\cC_{R,3} \cap \Sigma ) = 0$. However,  $\Vol (\cC_{2R,z_0} \cap \Sigma ) >  0$ for any $z_0 > 2$.  
\end{Exa}

In the remainder of this section, we will assume that 
\begin{align}	\label{e:assumehere}
	R^2 \geq 9 \, (n^2 + 1) {\text{ and }} \cC_{4R} \cap \Sigma \subset \{ x_{n+j} > 1 {\text{ for all }} j \geq 1 \} \, .
\end{align}
Define a function $h$ and domain $\Omega$ by 
\begin{align}
	h &= s^{2-n} - R^{2-n} \, \log \Pi_{j=1}^k x_{n+j} \, , \label{e:hagain} \\
	\Omega &= \{ h > (4\, R)^{2-n} \} \setminus \cC_R \, .
\end{align}

\begin{Lem}	\label{l:hn}
 On the set $\Omega$, we have $x_{n+j}< 3$ for all $j \geq 1$ and 
\begin{align}
	\dv_{\Sigma} \, \nabla  \, h \geq R^{-n} \,  \sum_{j=1}^k \, | \nabla^T x_{n+j}|^2 
	\, .
\end{align}
\end{Lem}

\begin{proof}
Using the lower bounds   for $h$ and $s$ on $\Omega$, we see that
\begin{align}
	(4\, R)^{2-n} < h =  s^{2-n} - R^{2-n} \, \log  \Pi_{j=1}^k x_{n+j}  \leq  R^{2-n} - R^{2-n} \, \log \Pi_{j=1}^k x_{n+j} \, .
\end{align}
This implies that 
\begin{align}
	\log  \Pi_{j=1}^k x_{n+j} \leq 1 - 4^{2-n} \, . 
\end{align}
This gives that  $ \Pi_{j=1}^k x_{n+j} < 3$. Since each is at least one, this gives the first claim.

Lemma \ref{l:sdoesk}
gives that
\begin{align}
		\dv_{\Sigma} \, \nabla  \, h &=  \dv_{\Sigma} \, \nabla  s^{2-n}  - R^{2-n} \,  \sum_{j=1}^k \dv_{\Sigma} 
			\, \left(  \frac{\nabla  x_{n+j} }{x_{n+j}} \right)  \notag \\
			 & \geq R^{2-n} \, \sum_{j=1}^k \,  \frac{| \nabla^T x_{n+j}|^2}{x_{n+j}^2} -(n-1) \, (n-2) \,  \sum_{j=1}^k \,  \frac{ | \nabla^T x_{n+j}|^2}{s^n}   
			  \\
		&\geq  \left( \frac{R^2}{9} - n^2 \right) \,   R^{-n} \, \sum_{j=1}^k \,   | \nabla^T x_{n+j}|^2
		 \, , \notag 
\end{align}
where the last inequality used that $s > R$ and $x_{n+j} < 3$ for $j=1, \dots$.  This gives the second claim since $R^2 \geq 9 \, (n^2 +1)$.
\end{proof}

We next define two cutoff functions on $\Omega$ as follows. First, let $\phi (s)$ be piecewise linear, zero for $s \leq R$ and one for $ s \geq \beta \, R$ where $\beta \in (1, 5/4)$ is fixed. Next, define  $\tau$ on $\Omega$ by $\tau^{2-n} = h$ and then let 
$\psi (\tau)$ be piecewise linear with $\psi = 1$ for $\tau \leq 3\, R$ and zero for $\tau \geq 4\, R$. 

\begin{Pro}	\label{p:useTL}
If $\Omega_{1} = \{ x \in \Omega \cap \Sigma \, | \, \psi \, \phi = 1 \}$, then 
\begin{align}
	\int_{\Omega_1} \sum_{j=1}^k |\nabla^T x_{n+j} |^2 &\leq C \, \Vol \, (\cC_{\beta \, R , 3} \cap \Sigma) \, , \\
	\int_{\Omega_1}  |\nabla^T h |^2 &\leq C \, \Vol \, (\cC_{\beta \, R , 3} \cap \Sigma) \, ,
\end{align}
where $C$ depends on $\beta , n, k$ and $R$.
\end{Pro}

\begin{proof}
The first claim follows by integrating $\dv_{\Sigma} \, \left( \phi (s) \, \psi (\tau) \, \nabla h \right)$, using Lemma \ref{l:hn} on the $\phi \, \psi \, \dv_{\Sigma} \, \nabla h$  term, and noting that the ``outer term'' $\phi \, \psi'(\tau) \, \langle \nabla^T \tau , \nabla^T h \rangle$ has the correct sign.  The second claim follows similarly by integrating 
$\dv_{\Sigma} \, \left( \phi (s) \, \psi (\tau) \, h \,  \nabla h \right)$.
\end{proof}

\begin{proof}[Proof of Theorem \ref{t:boundingoutB}]
Note that $\Omega_1$ extends further out in $s$ than $\beta \, R$, but that the ``height'' of $\Omega_1$ is lower than in $\cC_{\beta \, R , 3}$. 
We will show that we can bound $\Vol (\Omega_1) \leq C \, \Vol \, (\cC_{\beta \, R , 3} \cap \Sigma) $.  Once we have this, the theorem 
follows by scaling this result and iterating it.

Set $E = \sum_{j=1}^k |\nabla^T x_{n+j}|^2$. To bound the volume in $\Omega_1$, we start by using the first claim in Lemma \ref{l:sdoesk} to get
\begin{align}	\label{e:FTonR}
	\Vol (\Omega_1) \leq \int_{\Omega_1} |\nabla^T s|^2 + \int_{\Omega_1} E  \, .
\end{align}
The second integral is already bounded using the first claim in Proposition \ref{p:useTL}.  To bound the first, observe that
\begin{align}
	(2-n) \, s^{1-n} \, \nabla s = \nabla h + R^{2-n} \, \sum_{j=1}^k \frac{ \nabla x_{n+j}}{x_{n+j}} \, .
\end{align}
Since $1 < x_{n+j}$ and $R < s$  in $\Omega$, the  squared triangle inequality gives
\begin{align}
	(n-2)^2 \, |\nabla^T s|^2 \leq 2\, R^{2\, n - 2} \, \left( |\nabla^T h|^2 + R^{4-2n}\,  k \, E \right) \, .
\end{align}
Thus, the two integral bounds in  Proposition \ref{p:useTL} on $E$ and $ |\nabla^T h|^2$ 
give the desired  bound on the first term on the right in \eqr{e:FTonR}.
\end{proof}

\section{Bounds for surfaces}	\label{s:S6}

 The case of surfaces is  different because of parabolicity.  The previous sections focused on $n \geq 3$ and we explain here how to adapt the arguments to the case $n=2$ of surfaces. The parabolicity makes the case $n=2$ much easier.

\subsection{Volume doubling}

We will first explain how to extend Theorem \ref{t:doubling} to the case $n=2$.

\begin{Thm}  \label{t:doubling2}
There exists $\delta_0 > 0$   so that if 
 $\Sigma^2 \subset \RR^{2+k}$ is a proper stationary integral varifold  and 
\begin{align}	 
	\Sigma_{\frac{R}{2}, 4R  } \subset \{ |\Pi_2 (x)| \leq \delta_0 \, R\} \, , 
\end{align}
then 
 $ V(2R) \leq  75 \, V(R) $.
 \end{Thm}

\vskip1mm
Since the theorem is invariant under scaling, we have the freedom to choose $R$.  In the rest of this subsection, we will assume that $R \geq 8$. This will simplify some definitions and calculations.

\vskip1mm
We will use the following replacement for Lemma \ref{l:sdoesk} for surfaces (cf. \cite{CnKMR} for the codimension one case):

\begin{Lem}	\label{l:sdoesk2}
On $\Sigma$, we have that
\begin{align}
	|\nabla^{\perp} s|^2 & \leq  
	\sum_{j=1}^k   | \nabla^T x_{2+j}|^2 \, , \\
	 - \sum_{j=1}^k   | \nabla^T x_{2+j}|^2 &\leq s^2 \, \dv_{\Sigma} \, (\nabla  \, \log s ) \leq   \sum_{j=1}^k   | \nabla^T x_{2+j}|^2 \, .
\end{align}
\end{Lem}

\begin{proof}
Let $E= E(x) = \sum_{j=1}^k   | \nabla^T x_{2+j}|^2$ denote the tilt-excess at each point.
The first claim follows as in Lemma \ref{l:sdoesk} without change.   As in \eqr{e:110e}, we  have that
\begin{align}
	\dv_{\Sigma} \nabla s^2 = 4 - 2\, E \, ,
\end{align}
so we get 
\begin{align}
	s^2 \,    \dv_{\Sigma} \nabla \, \log s^2 =  s^2 \, \dv_{\Sigma} 
	\left( s^{-2} \, \nabla \,  s^2  \right) = 4 - 2\, E - 4 \, |\nabla^T s |^2 = 4 \, |\nabla^{\perp} s|^2 - 2 \, E \, .
\end{align}
The second claim follows from this and the first claim.
\end{proof}

This leads to  the following replacement for Lemma \ref{l:tiltgives}: 

\begin{Lem}	\label{l:tiltgives2a}
We have that
\begin{align}
	\int_{\Sigma_{2R}} |\nabla^T s|^2 \leq  25 \, V(R) + 2 \, \int_{\Sigma_{R,2R}} 
	 \sum_{j=1}^k |\nabla^T x_{2+j}|^2
	\, .
\end{align}
\end{Lem}

\begin{proof}
This follows as in Lemma \ref{l:tiltgives} with $s^{2-n}$ replaced by $\log s$. 
\end{proof}

Next we define $h$ by
\begin{align}
	h = \log s - 4 \, \frac{ |\Pi_2 (x)|^2}{R^2} \, .
\end{align}
Since $\log s$ is increasing, we now  subtract the $|\Pi_2 (x)|^2$ term in the definition of $h$ 
 (cf. \eqr{e:defh114}).

\vskip1mm
We have the following replacements for Lemma \ref{l:gg} and \ref{l:partway}:

\begin{Lem}	\label{l:gg2}
Given  $\epsilon >0$, there exists $\delta_0  > 0$  so that if  
  $|\Pi_2 (x)| \leq \delta_0 \, R$   in the region $R/2 \leq s \leq 4 R$, then in this region
\begin{align}
	\epsilon \,  \log s & \geq |\log s  - h |    \, , \\
	\dv_{\Sigma} \, (\nabla h) &\leq  - 4 \,  \frac{\sum_{j=1}^k   | \nabla^T x_{2+j}|^2}{ R^2} \, .
\end{align}
\end{Lem}

\begin{proof}
The proof follows  Lemma \ref{l:gg} with obvious modifications.  
\end{proof}

\begin{Lem}	\label{l:partway2}
If $\Sigma^2 \subset \RR^{2+k}$ is as in Theorem \ref{t:doubling2}, then 
\begin{align}
	 \int_{\Sigma_{R,2\, R}} \sum_{j=1}^k   | \nabla^T x_{2+j}|^2 &\leq
		  16 \, V(R)
	    \, .
\end{align}
\end{Lem}

\begin{proof}
The proof follows the proof of Lemma \ref{l:partway} with obvious modifications
(e.g., $\bar{s}$ is defined by $\log \bar{s} = h$).
\end{proof}

The proof of Theorem \ref{t:doubling2} now follows. 

\subsection{Euclidean volume growth for surfaces} 

We next prove Theorem \ref{t:EVG} for $n=2$ (cf. \cite{CnKMR} for surfaces in $\RR^3$).

\begin{Thm}	\label{t:EVGn2}
There exists $C$, so that for any $\alpha \in [0,1)$ there exists   $\delta = \delta (\alpha) > 0$ so that if   $R> 4$, 
$  \Sigma^2 \subset  \{ \frac{1}{4} < s < 4 \, R\}$ is a proper stationary integral varifold, and 
\begin{align}	\label{e:2p2n2}
	 \Sigma  \subset \{ |\Pi_2 (x)| \leq \delta  \, s^{\alpha} \} \, , 
\end{align}
then $V(R) \leq C \, R^2 \, V(1)$.
\end{Thm}

\vskip1mm
There are many examples of embedded minimal surfaces that have sublinear height growth; see, e.g., the survey \cite{HK} and 
the general construction in \cite{K}.

\vskip1mm
Using $\log s$ in place of $s^{2-n}$, we get the following analog of Lemma \ref{l:tiltgives2}:

\begin{Lem}	\label{l:tiltgives2n2}
If $\Sigma^2 \subset \RR^{2+k}$ is as in Theorem \ref{t:EVGn2} and $1\leq r \leq 2\, R$, then 
\begin{align}
	\int_{\Sigma_{1,r}} s^{-1} \,  |\nabla^T s|^2 \leq  12 \, r \, V(1) +   r \, \int_{\Sigma_{1,r}} 
	 \sum_{j=1}^k \frac{|\nabla^T x_{2+j}|^2}{s^2}
	\, .
\end{align}
\end{Lem}

This leads immediately to the following variation of Lemma \ref{l:usetilt}:

\begin{Lem}	\label{l:usetiltn2}
If $\Sigma^2 \subset \RR^{2+k}$ is as in Theorem \ref{t:EVGn2} and $ r \leq 2\, R$, then 
\begin{align}
	V(r) \leq \left( 1 + 12  \, r^2 \right) \, V(1) + 2 \, r^2 \, \int_{\Sigma_{1,r}} 
	 \sum_{j=1}^k \frac{|\nabla^T x_{2+j}|^2}{s^2} 
	\, .
\end{align}
\end{Lem}

The remaining ingredients in the proof of Euclidean volume growth go through without change to give Theorem \ref{t:EVGn2}.

\subsection{Rate of convergence}

Theorem \ref{t:rate} also extends to $n=2$:

    \begin{Thm}	\label{t:raten2}
There exist $ C$ and $R_0$ so that if 
$\Sigma^2 \subset \{ |\Pi_2 (x)| \leq 1 \}$ is a proper stationary integral varifold, then there exists an integer $\bV$ so that
for all $r \geq R_0$ and all $p$ with $p_{3} = \dots = p_{2+k}= 0$
\begin{align}		 
	(1-C\,  r^{-2}) \, \bV   \leq  \frac{\Vol (B_r (p)  \cap \Sigma)}{ \pi \, r^2}  \leq \bV \, .
\end{align}
 \end{Thm}
 
 \begin{proof}
 The proof follows the proof of Theorem \ref{t:rate}, but is much easier since the analog of 
 Theorem \ref{t:slab2} for $n=2$  follows immediately from Theorem \ref{t:EVGn2} and the reverse Poincar\'e inequality.
 \end{proof}

\subsection{One-sided bounds}

Finally, we explain the modifications necessary to extend Theorem \ref{t:boundingoutB} to the case $n=2$. 
The changes are parallel to what we have done already, with $\log s$ replacing $s^{2-n}$.  As for $n \geq 3$, we begin by translating
 in space so that $x_{2+j} > 1$ for each $j\geq 1$.

First, we modify the definition of
$h$ given in \eqr{e:hagain}, by defining
\begin{align}
	h &= \log s  +   \log \Pi_{j=1}^k x_{2+j} \, .
\end{align}
We will assume that $R \geq 8$ and define $\Omega $ to be the set where $s > R$ and $h < \log (4 \, R)$.  We then get the following replacement for Lemma \ref{l:hn}:

\begin{Lem}	\label{l:hn2}
 On the set $\Omega$, we have $x_{2+j}< 4$ for all $j \geq 1$ and 
\begin{align}
	\dv_{\Sigma} \, \nabla  \, h \leq  - \frac{1}{32}  \,  \sum_{j=1}^k \, | \nabla^T x_{2+j}|^2 
	\, .
\end{align}
\end{Lem}

\begin{proof}
This follows as in the proof of Lemma \ref{l:hn} with Lemma \ref{l:sdoesk2} replacing Lemma \ref{l:sdoesk}.
\end{proof}

The remainder of the proof now follows with the obvious modifications.


\begin{thebibliography}{A}  
   
    \bibitem[Al]{Al}
W. K.  Allard, \emph{On the first variation of a varifold}. Ann. of Math. (2) 95 (1972), 417--491.

   
   \bibitem[Be]{Be}
   C. 
   Bellettini,  {\it Extensions of Schoen-Simon-Yau and Schoen-Simon theorems via iteration \`a la De Giorgi},
    Invent. Math. 240 (2025), no. 1, 1--34.

\bibitem[BB1]{BB1}
 J. Bernstein and C. Breiner,  {\it Conformal structure of minimal surfaces with finite topology},
  Comment. Math. Helv. 86 (2011), no. 2, 353--381.

 \bibitem[BB2]{BB2}
 J. Bernstein and C. Breiner,  {\it Helicoid-like minimal disks and uniqueness},
  J. Reine Angew. Math. 655 (2011), 129--146.

\bibitem[B]{B}
E. Bishop, {\it Conditions for the analyticity of certain sets}, Michigan Math. J. 11 (1964), 289--304.
   
   \bibitem[Bl]{Bl}
D. Blair, {\it A generalization of the catenoid}, Canad. J. Math. 27 (1975), 231--236. 


 \bibitem[BDM]{BDM}
 E.  Bombieri, E.  De Giorgi and M. Miranda, \emph{Una maggiorazione a priori relativa alle ipersuperfici minimali non parametriche}, Arch. Ration. Mech. Anal. 32 (1969), 255--267. 

   
\bibitem[CaP]{CaP}
X. Cabr\'e and G. Poggesi, 
{\it Stable solutions to some elliptic problems: minimal cones, the Allen-Cahn equation, and blow-up solutions},
Geometry of PDEs and related problems, 1--45,
Lecture Notes in Math., 2220, Fond. CIME/CIME Found. Subser., Springer, Cham, 2018.

 
 \bibitem[CfNS]{CfNS}
    L. Caffarelli, L. Nirenberg,  and J. Spruck, 
{\it On a form of Bernstein's theorem},
 Analyse math\'ematique et applications, 55--66, Gauthier-Villars, Montrouge, 1988.
 
\bibitem[Cj]{Cj}
J. Choe,   
{\it On the existence of higher-dimensional Enneper's surface},
Comment. Math. Helv. 71 (1996), no. 4, 556--569.

\bibitem[CjH]{CjH}
J. Choe and J. Hoppe,  
{\it Higher dimensional minimal submanifolds generalizing the catenoid and helicoid},
 Tohoku Math. J. (2) 65 (2013), no. 1, 43--55. 

\bibitem[CMM]{CMM} T.H. Colding, F. Mart\'in, and W.P. Minicozzi II, 
{\it  Minimal surfaces with rapid area growth},
Pure and Applied Math. Quarterly, to appear.
 
\bibitem[CM1]{CM1} T.H. Colding and W.P. Minicozzi II, 
A course in minimal surfaces, Graduate Studies in Mathematics, vol. 121 (American Mathematical Society, Providence, RI, 2011). 

\bibitem[CM2]{CM2} T.H. Colding and W.P. Minicozzi II, 
{\it The space of embedded minimal surfaces of fixed genus in a 3-manifold. I. Estimates off the axis for disks},
 Ann. of Math. (2) 160 (2004), no. 1, 27--68.

\bibitem[CM3]{CM3} T.H. Colding and W.P. Minicozzi II, 
{\it The space of embedded minimal surfaces of fixed genus in a 3-manifold. II. Multi-valued graphs in disks},
 Ann. of Math. (2) 160 (2004), no. 1, 69--92.

\bibitem[CM4]{CM4} T.H. Colding and W.P. Minicozzi II, 
{\it The space of embedded minimal surfaces of fixed genus in a 3-manifold. III. Planar domains},
 Ann. of Math. (2) 160 (2004), no. 2, 523--572.

\bibitem[CM5]{CM5} T.H. Colding and W.P. Minicozzi II, 
{\it The space of embedded minimal surfaces of fixed genus in a 3-manifold. IV. Locally simply connected},
Ann. of Math. (2) 160 (2004), no. 2, 573--615.

\bibitem[CM6]{CM6} T.H. Colding and W.P. Minicozzi II, 
{\it The space of embedded minimal surfaces of fixed genus in a 3-manifold V; fixed genus},
 Ann. of Math. (2) 181 (2015), no. 1, 1--153.

\bibitem[CM7]{CM7} T.H. Colding and W.P. Minicozzi II, 
{\it Distance between minimal surfaces and  flows}, 
Arch. Ration. Mech. Anal. 250 (2026), no. 5, Paper No. 92.   

\bibitem[CM8]{CM8} T.H. Colding and W.P. Minicozzi II, 
{\it Liouville theorem for immersed minimal surfaces in any codimension}, preprint.  

\bibitem[CM9]{CM9} T.H. Colding and W.P. Minicozzi II, 
{\it Height of minimal surfaces: sharp thresholds}, preprint.  

\bibitem[CM10]{CM10} T.H. Colding and W.P. Minicozzi II, 
{\it A dichotomy for minimal submanifolds}, preprint.  



\bibitem[CnKMR]{CnKMR}
P. Collin,  R. Kusner,  W. Meeks,  and H. Rosenberg,  
{\it The topology, geometry and conformal structure of properly embedded minimal surfaces},
J. Differential Geom. 67 (2004), no. 2, 377--393. 

\bibitem[CGMR]{CGMR}
G. Colombo, E.S. Gama, L. Mari, and M. Rigoli,
\emph{Nonnegative Ricci curvature and minimal graphs with linear growth}. Anal. PDE 17 (2024), no. 7, 2275--2310.

\bibitem[CMMR]{CMMR}
G. Colombo, M. Magliaro, L. Mari, and M. Rigoli, 
\emph{Bernstein and half-space properties for minimal graphs under Ricci lower bounds}. 
Int. Math. Res. Not. IMRN 2022, no. 23, 18256-18290.



\bibitem[dG]{dG}
E. De Giorgi,
{\it Sulla differenziabilit\`a e l'analiticit\`a delle estremali degli integrali multipli regolari},
Mem. Accad. Sci. Torino. Cl. Sci. Fis. Mat. Nat. (3) 3 (1957), 25--43. 


\bibitem[dL]{dL}
C. De Lellis, 
{\it Allard's interior regularity theorem: an invitation to stationary varifolds},
 Nonlinear analysis in geometry and applied mathematics. Part 2, 23--49,
Harv. Univ. Cent. Math. Sci. Appl. Ser. Math., 2, Int. Press, Somerville, MA, 2018.

\bibitem[dLHS]{dLHS}
C. De Lellis, J. Hirsch, and L. Spolaor,
{\it Stationary and stable varifolds with singularities}, preprint, 
arXiv/2509.21508.

\bibitem[D1]{D1}
Q. Ding, 
\emph{Poincar\'e inequality on minimal graphs over manifolds and applications}. 
Camb. J. Math. 13 (2025), no. 2, 225-299.


\bibitem[D2]{D2}
Q. Ding, 
\emph{Liouville theorem for minimal graphs over manifolds of nonnegative Ricci curvature}. 
Anal. PDE 18 (2025), no. 10, 2537-2550.

\bibitem[EH]{EH}
K. Ecker and G. Huisken, {\it A Bernstein result for minimal graphs of controlled growth}, J. Differential Geom. 31 (1990), 397--400.

\bibitem[ERM]{ERM}
N. Edelen, L.A.F. Reyna, and P. Minter,
{\it Entire area-minimizing surfaces in $\RR^4$ are algebraic}, preprint.

\bibitem[FPa]{FPa}
S. Fakhi and F. Pacard,  
{\it Existence result for minimal hypersurfaces with a prescribed finite number of planar ends},
Manuscripta Math. 103 (2000), no. 4, 465--512.



\bibitem[F]{F}
H. Federer,  \emph{Geometric measure theory}. 
Die Grundlehren der mathematischen Wissenschaften, Band 153, Springer-Verlag New York Inc., New York 1969.






\bibitem[HK]{HK} 
D. Hoffman and H. Karcher, \emph{Complete embedded minimal surfaces of finite total curvature}. 
Geometry, V, 5-93, Encyclopaedia Math. Sci., 90, Springer, Berlin, 1997.

\bibitem[HM]{HM}  D. Hoffman and W.H. Meeks III, 
\emph{The strong halfspace theorem for minimal surfaces}.  
Invent. Math. 101 (1990), no. 2, 373--377.


 




\bibitem[JM]{JM}
 L.P. Jorge and W.H. Meeks  III, 
 {\it The topology of complete minimal surfaces of finite total Gaussian curvature},
  Topology 22 (1983), no. 2, 203--221. 
  
  \bibitem[JX]{JX}  
L.P. Jorge and F. Xavier, 
{\it A complete minimal surface in $\RR^3$ between two parallel planes},
 Ann. of Math. (2) 112 (1980), no. 1, 203--206.
 
\bibitem[KPa]{KPa}  
S. Kaabachi and F. Pacard, 
\emph{Riemann minimal surfaces in higher dimensions}. 
J. Inst. Math. Jussieu 6 (2007), no. 4, 613--637.

\bibitem[K]{K}
N. Kapouleas, 
{\it Complete embedded minimal surfaces of finite total curvature},
J. Differential Geom. 47 (1997), no. 1, 95--169.


  
  




 \bibitem[LSW]{LSW}
W. Littman,  G. Stampacchia,  and H. Weinberger,  
{\it Regular points for elliptic equations with discontinuous coefficients},
Ann. Scuola Norm. Sup. Pisa Cl. Sci. (3) 17 (1963), 43--77. 




\bibitem[MP]{MP}
W.H.  Meeks III and J. P\'erez, 
{\it The classical theory of minimal surfaces},
 Bull. Amer. Math. Soc. (N.S.) 48 (2011), no. 3, 325--407.

\bibitem[MPR]{MPR}
W.H.  Meeks III, J. P\'erez and A. Ros, 
{\it Properly embedded minimal planar domains},
 Ann. of Math. (2) 181 (2015), no. 2, 473--546.
 
 
 \bibitem[M]{M} M.J. Micallef, \emph{Stable minimal surfaces in Euclidean space}. J. Differential Geom. 19 (1984), no. 1, 57--84.

 
 


\bibitem[Mo1]{Mo1}
J.  Moser,  {\it{A new proof of De Giorgi's theorem concerning the regularity problem for elliptic differential equations}},
Comm. Pure Appl. Math. 13 (1960), 457--468. 

\bibitem[Mo2]{Mo2}
J.  Moser,  {\it{On Harnack's theorem for elliptic differential equations}}, Comm. Pure Appl. Math. 14 (1961), 577--591. 

\bibitem[Os]{Os}
R. 
Osserman,
{\it Global properties of minimal surfaces in $E^3$  and $E^n$}, 
Ann. of Math. (2) 80 (1964), 340--364.


\bibitem[P1]{P1}
J. P\'erez,  {\it A new golden age of minimal surfaces},
 Notices Amer. Math. Soc. 64 (2017), no. 4, 347--358. 
 
 \bibitem[P2]{P2}
J. P\'erez,  {\it Minimal surfaces of finite genus: Classification, dynamics and laminations},
ICM Proceedings 2026, to appear.

\bibitem[R]{R}
H. Rutishauser, 
{\it \"Uber Folgen und Scharen von analytischen und meromorphen Funktionen mehrerer Variabeln, sowie von analytischen Abbildungen},
Acta Math. 83 (1950), 249--325.

\bibitem[Sc]{Sc}
R. Schoen,
{\it Uniqueness, symmetry, and embeddedness of minimal surfaces},
J. Differential Geom. 18 (1983), no. 4, 791--809. 

\bibitem[SS]{SS}
R. Schoen and L. Simon,  {\it Regularity of stable minimal hypersurfaces},
 Commun. Pure Appl. Math. 34, 741--797 (1981).  
 



\bibitem[Si]{Si} 
L. Simon, 
\emph{Lectures on geometric measure theory}, Proceedings of the Centre for Mathematical Analysis, Australian National University, vol. 3, 1983.

\bibitem[S]{S}  
W. Stoll, \emph{The growth of the area of a transcendental analytic set}. I, II. Math. Ann. 156 (1964), 144-170.




\bibitem[Tr]{Tr}
N. Trudinger,  
{\it A new proof of the interior gradient bound for the minimal surface equation in n dimensions},
Proc. Nat. Acad. Sci. U.S.A. 69 (1972), 821--823. 

 
 \bibitem[W]{W}
N.  Wickramasekera, {\it A general regularity theory for stable codimension 1 integral varifolds},
 Ann. of Math. (2) 179 (2014), no. 3, 843--1007.

 
\bibitem[Wi]{Wi} W. Wirtinger, \emph{Eine determinantenidentit\"at und ihre anwendung auf analytische gebilde und
Hermitesche massbestimmung}, Monatsh. Math. Physik 44 (1936), 343--365.
 

\end{thebibliography}
\end{document}